\documentclass[10pt,a4paper,english]{article}
\usepackage[utf8]{inputenc}
\usepackage[english]{babel}
\usepackage{amsmath}
\usepackage{amsfonts}
\usepackage{amssymb}
\usepackage{amsthm}
\usepackage{pgfplots}
\usepackage{graphicx}
\usepackage{circuitikz}
\usepackage{tikz-cd}
\usepackage[justification=centering]{caption}
\graphicspath{ {./figures/} }
\usepackage{geometry}
\usepackage{placeins}
\usepackage[backend=biber,useprefix=true,maxnames=7,minnames=1,maxcitenames=7,mincitenames=1,minalphanames=1,
maxalphanames=7,sorting=nyt,style=alphabetic]{biblatex}
\renewbibmacro{in:}{}
\usetikzlibrary{intersections, pgfplots.fillbetween}
\usetikzlibrary{patterns}
\usetikzlibrary{shapes,decorations}
\pgfplotsset{compat=1.16}
\newtheorem{Theorem}{Theorem}[section]
\newtheorem*{introthm}{Theorem}

\newtheorem{Coro}[Theorem]{Corollary}

\theoremstyle{definition}
\newtheorem{Definitio}[Theorem]{Definition}
\newenvironment{Definition}{\begin{Definitio} \rmfamily }{\end{Definitio}}
\newtheorem{Lemma}[Theorem]{Lemma}
\newtheorem{Proposition}[Theorem]{Proposition}
\newtheorem{Remar}[Theorem]{Remark}
\newenvironment{Remark}{\begin{Remar} \rmfamily }{\end{Remar}}
\newtheorem{Exampl}[Theorem]{Example}
\newenvironment{Example}{\begin{Exampl} \rmfamily }{\end{Exampl}}
\author{Rudy Dissler}
\title{Relative multisections of higher-dimensional manifolds with boundary}
\begin{document}
\maketitle
\begin{abstract}
A multisection, as defined by Ben Aribi, Courte, Golla and Moussard \cite{aribi2023multisections}, is a decomposition of a closed orientable manifold into top-dimensional $1$--handlebodies, whose subcollections intersect along \mbox{$1$--handlebodies} of lower dimensions, and whose global intersection is a closed surface. This concept extends the notions of Heegaard splittings and trisections to higher dimensions. In this article, we adapt multisections to compact manifolds with boun\-dary, generalizing sutured Heegaard splittings and relative trisections to eve\-ry dimension. We provide a diagrammatic approach, which encompasses sutured Heegaard diagrams and relative trisection diagrams. A relative multisection induces a particular decomposition of the boun\-dary of the manifold, which we call a relative fibration. We show that, for $n \geq 4$, a connected $n$--manifold which relatively fibers is necessarily $S^n$ or a connected sum of copies of~$S^1 \times S^{n-1}$. We prove that a compact $5$--manifold whose boundary relatively fibers admits a relative multisection. We also state a gluing theorem that allows to combine suitable relatively multisected manifolds with boundary into multisected closed manifolds.   
\end{abstract}
\section{Introduction}
Throughout this article, we consider compact, orientable, connected manifolds which are smooth or PL. The notions are developed mostly in the smooth category. Their adaptation to the PL category is mentioned when needed. A Heegaard splitting is a decomposition of a closed $3$--manifold into two \mbox{$3$--dimensional} \mbox{$1$--handlebodies}, intersecting along their common boundary. This concept was adapted to dimension $4$ by \textcite{gay2016trisecting}: a trisection is a decomposition of a closed $4$--manifold into three \mbox{$4$--dimensional} $1$--handlebodies, intersecting pairwise along $3$--dimensional \mbox{$1$--handlebodies}, and with common intersection a connected closed surface. Generalizing further, \textcite{aribi2023multisections} define a multisection, or $n$--section (with $n \geq 2$) as a decomposition of a closed $(n+1)$--manifold into~$n$ \mbox{$1$--handlebodies}, such that all the pieces of the decomposition intersect along a connected closed surface, and $k<n$ pieces intersect along an \mbox{$(n-k+2)$--dimensional} handlebody. 
\par In dimension $3$ and $4$, these theories have been adapted to compact manifolds with boun\-dary. Gabai \cite{gabai1983foliations} introduced sutured $3$--manifolds, from which Goda defined sutured Heegaard splittings (\cite{goda1992heegaard}, see \textcite{juhasz2006holomorphic} for the diagrammatic approach). Castro, Gay, Kirby and Pinz{\'o}n-Caicedo developed $4$--dimensional relative trisections in \cite{gay2016trisecting,castro2016relative,castro2018diagrams,castro2018trisections}. But the adaptation of $n$--sections to $(n+1)$--manifolds with boundary, for $n \geq 4$, has not been investigated yet. This is the purpose of this article.
\par To define relative multisections, we use compression bodies, i.e. connected cobordisms using only $1$--handles. A relative $n$--section is a decomposition of a compact $(n+1)$--manifold with boundary $W$ into $n$ compression bodies, such that all the pieces of the decomposition intersect along a connected compact surface with non-empty boundary, and $k<n$ pieces intersect along an \mbox{$(n-k+2)$--dimensional} compression body. Moreover, we require that the intersection of such a compression body with $\partial W$ is the product of a ball with a given compact surface with no closed component. This definition encompasses sutured Heegaard splittings and relative trisections.
\par Sutured Heegaard splittings and relative trisections induce specific decompositions of the boun\-dary of the decomposed manifold: respectively a $2$--dimensional balanced sutured decomposition and a tridimensional open-book decomposition. In higher dimensions, relative $n$--sections are designed to induce a similar structure on the boundary of the multisected manifold. This decomposition of the boundary, which we call a relative fibration, is a fibration over~$S^{n-2}$ of the complement of a $1$--dimensional link, with fiber the interior of a compact surface, called the page of the fibration. In dimension $3$ and $4$, given such decompositions of the boundary of a manifold, there exists a matching sutured Heegaard splitting or relative trisection. We show that this result is true in dimension $5$ as well.
\begin{introthm}[Theorem \ref{thm_relative_multisection}]
\label{thm_existence_dim5}
Let $W$ be a smooth, compact, oriented manifold of dimension~$5$, such that $\partial W$ admits a relative fibration. Then there exists a relative $4$--section of $W$ inducing this relative fibration. 
\end{introthm}
Note that, because all $5$--dimensional PL manifolds are smoothable, this theorem remains true in the PL category. 
\par Any closed manifold of dimension $2$ or $3$ admits a relative fibration. However, we show that, in higher dimensions, the picture is very different. 
\begin{introthm}[Theorem \ref{prop_rel_trivial}]
\label{thm_rel_fib_trivial}
A smooth (resp. PL), closed, connected, oriented manifold of dimension $n \geq 4$, which admits a relative fibration, is necessarily diffeomorphic (resp. PL-homeomorphic) to $S^n$ or to a connected sum of copies of~$S^1 \times S^{n-1}$.
\end{introthm} 
Therefore, if a compact $(n+1)$--manifold, with $n \geq 4$, admits a relative multisection, then its boundary is a disconnected union of spheres and connected sums of copies of~$S^1 \times S^{n-1}$. We can imagine a more general adaption of multisections to manifolds with boundary, which would alleviate this restriction on the boundary. This is the purpose of a forthcoming article, where we develop a concept of \emph{pseudo-multisections}. However, relative multisections are interesting constructions, because they are designed to preserve many of the features of relative trisections. We are able to produce interesting classes of examples of relatively multisected manifolds, which we can describe diagrammatically, and combine into multisected manifolds, as we shall see in the following paragraphs.     
\par In dimension $3$ and $4$, there exists a diagrammatic approach to sutured Heegaard splittings and relative trisections, via sutured Heegaard diagrams and relative trisection diagrams. We extend this approach to relative multisections, defining relative multisection diagrams. A relative \mbox{$n$--section} diagram associated to a given relative $n$--section is a tuple $\lbrace \Sigma, (\alpha^i)_{1 \leq i \leq n} \rbrace$, where $\Sigma$ is a surface corresponding to the intersection of all the pieces, and every $\alpha^i$ is a set of closed curves on $\Sigma$, defining one tridimensional intersection (i.e. some intersection of $n-1$ pieces). We can also define an \emph{abstract} relative $n$--section diagram, by specifying conditions on $n$ families of curves on a compact surface with boundary. We obtain the following correspondence between relative multisections and abstract relative multisection diagrams.
\begin{introthm}[Theorem \ref{thm_rel_diagram}]
\label{thm_rel_diagram_intro}
For $2 \leq n \leq 6$, every abstract relative $n$--section diagram is the diagram of some smooth multisected $(n+1)$--manifold. If $n \leq 5$, this manifold is unique up to multisection-preserving diffeomorphism. 
\par For every $n \geq 2$, every abstract relative $n$--section diagram is the diagram of some PL multisected $(n+1)$--manifold, and this manifold is unique up to multisection-preserving PL-homeomorphism.  
\end{introthm}
\par We can glue two sutured Heegaard splittings or two relative trisections along their boundaries, provided that the decompositions induced on these boundaries are diffeomorphic. The result is a Heegaard splitting or a trisection of the closed manifold thus obtained. We show that this is true in higher dimensions as well. We describe the construction of a multisection diagram associated to a multisection obtained by gluing. We provide examples of such constructions. In particular, we use this gluing method to obtain a multisection of a spun manifold of dimension $n+1$, whose $n$--dimensional fiber admits a multisection.   
\begin{introthm}[Theorem \ref{thm:gluing}]
\label{gluing_theorem}
Let $W$ and $W'$ be two smooth (resp. PL), compact manifolds of dimension $n+1$, that admit relative multisections $W= \cup_{i=1}^n W_i$ and $W'=\cup_{i=1}^n W_{i}'$. Suppose that there exists a diffeomorphism (resp. PL-homeomorphism) $f: \partial W \rightarrow \partial W'$, compatible with the relative fibrations on $\partial W$ and $\partial W'$, and such that $f(W_i \cap \partial W ) = W_{i}' \cap \partial W'$ for every $i$.\\ 
Then $W \cup_{f} W' = \cup_{i=1}^n (W_i \cup_{f} W_i')$ is an $n$--section of $W \cup_{f} W'$.   
\end{introthm}
\par This article is organized as follows. Section \ref{section_definition} starts with a quick introduction to submanifolds with corners of a smooth manifold, a notion which will be needed later on to define the pieces of a smooth multisection. Then, we recall the definition of a multisection of a closed manifold, as well as some basic facts and examples, from \cite{aribi2023multisections}. In Section~\ref{sec_rel_mult}, we introduce relative multisections and relative fibrations. In particular, we give a proof of Theorem \ref{prop_rel_trivial}. Section~\ref{section_diagram} is devoted to relative multisection diagrams. In Section~\ref{section_gluing}, we show how to glue relative multisections and relative multisection diagrams. In Section~\ref{section_existence}, we give the existence proof of multisections in dimension $5$ (Theorem~\ref{thm_relative_multisection}), using Morse Theory. In Section~\ref{section_more_examples}, we construct more complicated examples of relative multisections, and provide more applications of the gluing method. 
\par We use the following conventions. For any family of sets $\lbrace S_i \rbrace_{i \in I}$, we denote by $S_I$ the intersection~$\cap_{i \in I} S_i$. To make notations less cluttered, we denote by $S_{i_1...i_k}$ the intersection~$S_{\lbrace i_1,...,i_k \rbrace}$. We denote by $\Sigma_{g,b}$ a genus--$g$ surface with $b$ boundary components. We refer to $1$--handlebodies as handlebodies.\\
\par The author wishes to thank his advisors, Benjamin Audoux and Delphine Moussard, for their help and many meaningful discussions.
\section{Background} 
\label{section_definition}
\subsection{Submanifold with corners}
Throughout this article, we decompose smooth manifolds (with or without boundary) into collections of submanifolds with corners. To that end, we recall the definition of a submanifold with corners of a smooth manifold, as well as some basic properties. We start with closed manifolds, then consider the case of manifolds with non-empty boundary.
\begin{Definition}
\label{def_corner}
Let $W$ be a smooth closed manifold, together with a smooth atlas $\mathcal{U}$. Let~$Y$ be a topological $m$--dimensional submanifold of $W$. We say that $Y$ is a \emph{submanifold with corners} of~$W$ if for all $y \in Y$, there exists an integer $k \geq 0$, and a chart $(y \in U_y, \phi_y)$ in $\mathcal{U}$, with $\phi_y(y) = \lbrace 0 \rbrace$, such that $\phi_y (U_y \cap Y) = \mathbb{R}^{m-k} \times [0, \infty)^k$. We say that such a $y$ belongs to the \emph{$(m-k)$--stratum} of~$Y$.  
\end{Definition}
\begin{Remark}
Because there cannot exist a diffeomorphism of $\mathbb{R}^m$ sending $\mathbb{R}^{m-k} \times [0, \infty)^k$ to $\mathbb{R}^{m-\ell} \times [0, \infty)^{\ell}$ if $k \neq \ell$, each $y$ belongs to one and only one stratum of $Y$. Therefore we have a partition of $Y$ into its strata, which we call a \emph{stratification} of $Y$.  
\par The boundary of a submanifold with corners $Y \subset W$ is well-defined as its topological boundary or, equivalently, as the union of all its strata but the top-dimensional one. The interior of $Y$ corresponds to its top-dimensional stratum. It is an open smooth submanifold of $W$.   
\end{Remark}
We want to define a way to smoothen corners. By that, we mean defining a smooth submanifold with boundary, which corresponds to a given submanifold with corners. Me\-thods for smoothening (or rounding, straightening) corners already exist in the literature (see Milnor \cite[p.~86,87]{milnor2007differentiable}, see also \cite[Chapter 6, paragraph 8]{kosinski2013differential}). But these references deal with specific types of corners, such as the corners resulting from trivial cobordisms and handle attachments. We need to work in a slightly more general setting, although the arguments are very similar. Moreover, our definition will allow us to see the smoothing of a submanifold with corners as a subset of this submanifold.     
\begin{Proposition}
\label{prop_smoothing}
Given a compact $m$--submanifold with corners~$Y$ of a smooth manifold~$W$, there exists a smooth submanifold with boundary of~$W$, which is a subset of $Y$ and is (continuously) ambient isotopic to $Y$. Moreover, the isotopy is smooth on each stratum of~$Y$.    
\end{Proposition} 
\begin{proof}
We start by proving the existence of a smooth vector field $\chi$ on $W$, pointing inside~$Y$ on~$\partial Y$, and transverse to each stratum of $\partial Y$. Denote by $\lbrace (U_{y_i}, \phi_i) \rbrace_{1 \leq i \leq \ell}$ a finite open cover of $Y$, by charts as in Definition~\ref{def_corner}. Let each $y_i$ belong to the $(m-k_i)$--stratum of $Y$, and let~$U_i = U_{y_i}$ to simplify notations. Denote by $(\rho_i)_{1 \leq i \leq \ell}$ a subordinate partition of unity. For every $i$, there exists a smooth vector field $\epsilon_i$ on $U_i$, which is inward on $U_i \cap \partial Y$. A linear combination of such vector fields, with non-negative (and at least one non-zero) coefficients, shares the same property. Define a smooth vector field on $W$ by $\chi_i(x) = \rho_i(x) \epsilon_i(x)$ if $x \in U_i$ and $\chi_i \equiv 0$ otherwise. Then define a vector field $\chi$ by $\chi(x) = \Sigma_i \chi_i(x)$ for $x \in W$: this vector field has the required properties. This implies that $\partial Y$ admits a collar in $Y$, i.e. a subset of $Y$ diffeomorphic to $\partial Y \times [0,1)$, with~$\partial Y$ identified with~$\partial Y \times \lbrace 0 \rbrace$. Let $0 <t<1$. We smoothen the corners of $Y_t = Y \setminus (\partial Y \times [0,t))$ inside~$Int(Y)$: this is done as follows. 
\par A component of the $1$--stratum of $Y_t$ is either a circle, or an open interval $A$ with $\partial \overline{A}$ consis\-ting in either one or two points in the $0$--stratum of $Y$. Suppose that $\partial \overline{A}$ consists in two points in the \mbox{$0$--stratum} of $Y$. By pushing along $\chi$, we find a neighborhood of~$\overline{A}$ in~$Y$ diffeomorphic to~$\overline{A}  \times [0,\infty)^{m-1}$. We isotope $[0, \infty)^2 \subset [0,\infty)^{m-1}$ on a subset of $[0, \infty)^2$ bounded by the graph of a smooth function~$f$, taking values in $(0, \infty)$, as in Figure \ref{fig_smooth}. This isotopy is smooth everywhere but at the origin, and induces a piecewise smooth isotopy (smooth everywhere but on~$\overline{A}$) between~$\overline{A} \times [0,\infty)^{m-1}$ and a subset of~$\overline{A} \times [0,\infty)^{m-1}$ diffeomorphic to $\overline{A} \times \mathbb{R} \times [0, \infty)^{m-2}$. The image of~$Y_t$ by this isotopy is a subset of $Y_t$, piecewise smoothly isotopic to~$Y_t$, and a submanifold with corners of~$W$, whose $0$--stratum has $2$ components fewer than the $0$--stratum of $Y_t$. Iterating this procedure straightens the unions of closures of components of the $1$--stratum bounded by two points in the $0$--stratum, producing components of the $1$--stratum whose closure is bounded by one point in the $0$--stratum. Then we apply the same operation to the components of the \mbox{$1$--stratum} bounded by one point in the $0$--stratum, until we obtain a submanifold with corners of $W$, piecewise smoothly isotopic to~$Y_t$, with an empty $0$--stratum. The $2$--stratum of this submanifold consists in smooth closed surfaces, or open surfaces whose closure is bounded by components of the $1$--stratum. Again by the same argument, we isotope the $1$-- and $2$--strata so as to obtain a submanifold with corners inside $Y_t$, with an empty $1$--stratum. We iterate this procedure until we obtain a smooth submanifold with boundary, included in $Y_t$ and piecewise smoothly isotopic to $Y_t$. Finally we extend the isotopy trivially to $W$.  
\par Note that the isotopy defined at each step is smooth on each stratum of $Y_t$. Moreover, the image of this isotopy depends only (up to smooth isotopy) on the graph of the smooth function $f$ from $(0, \infty)$ to $(0, \infty)$. As any two graphs of such functions are smoothly isotopic in $(0, \infty)^2$ (they can be compactified as closed curves in a disk, which is simply connected), the choice of the graph only changes the resulting submanifold up to a smooth isotopy.   
\end{proof}
\begin{figure}[h!]
\[
\begin{tikzpicture}[scale=0.3]
\node (mypic) at (0,0) {\includegraphics[scale=0.5]{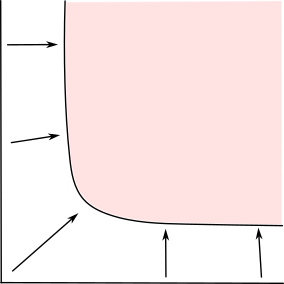}};
\end{tikzpicture}
\]
\caption{Isotopy between $[0,\infty)^2$ and a subset of $[0,\infty)^2$ diffeomorphic to $\mathbb{R} \times [0, \infty)$}
\label{fig_smooth}
\end{figure}
The \emph{smoothing} of a submanifold with corners $Y \subset W$ is defined (up to ambient isotopy) as the smooth submanifold $s(Y) = F(Y,1) \subset Y$, with $F$ an ambient isotopy of $W$ as in Proposition \ref{prop_smoothing}. For any subset $S \subset Y$, we denote by $s(S)$ the image $F(S,1)$.
\par Now we consider the case of a submanifold with corners of a smooth manifold with non-empty boundary. We use a slightly different definition for the strata of such a submanifold, which we call \emph{relative strata}. 
\begin{Definition}
\label{corners_withbdy}
Let $W$ be a smooth $(n+1)$--manifold with boundary, together with a smooth atlas with boundary $\mathcal{A}$. Let $Y$ be a topological $m$--dimensional submanifold of $W$. We say that $Y$ is a \emph{submanifold with corners} of $W$ if, for all $y \in Y$:
\begin{itemize}
\item either $y \in \partial W$, and there exists an integer $k \geq 0$ and a chart $(U_y, \phi_y)$ in $\mathcal{A}$, with $\phi_y(y) = \lbrace 0 \rbrace$, such that $\phi_y(U_y)= \mathbb{R}^{n} \times [0,1)$ and $\phi_y (U_y \cap Y) = (\mathbb{R}^{m-k-1} \times [0, \infty)^{k}) \times  [0,1)$;
\item or $y \in W \setminus \partial W$, and there exists an integer $k \geq 0$ and a chart $(U_y, \phi_y)$ in $\mathcal{A}$, with $\phi_y(y) = \lbrace 0 \rbrace$, such that $\phi_y (U_y \cap Y) = \mathbb{R}^{m-k} \times [0, \infty)^{k}$.
\end{itemize}  
We say that such a $y$ belongs to the \emph{relative $(m-k)$--stratum} of $Y$. 
\end{Definition}
\begin{Remark}
As in the closed setting, each $y$ belongs to one and only one relative stratum of $Y$. Therefore we have a partition of~$Y$ into its relative strata, which we call a \emph{relative stratification} of~$Y$ (see Figure \ref{fig_srata} for a $3$--dimensional example). A relative stratification differs from a stratification as usually defined, where a point on the boundary and a point in the interior of the manifold necessarily belong to different strata, which are smooth submanifolds without boundary. In Definition \ref{corners_withbdy}, the relative strata are smooth manifolds with boundary, neatly embedded in $W$. The intersection of a relative $(m-k)$--stratum of $Y$ with $\partial W$ is an $(m-k-1)$--dimensional submanifold of $\partial W$. Considered as a submanifold with corners of $\partial W$, $Y \cap \partial W$ inherits a stratification from the relative stratification of $Y$.
\end{Remark}
\begin{figure}[h!]
\[
\begin{tikzpicture}[scale=0.3]
\node (mypic) at (0,0) {\includegraphics[scale=0.5]{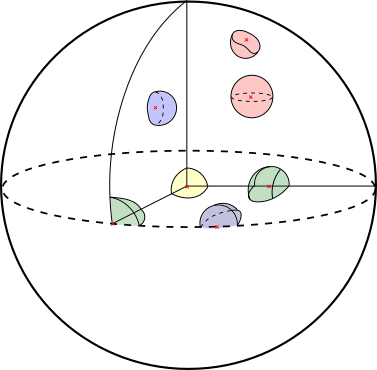}};
\node (a) at (-1,0.8) {$x$};
\node (a) at (-2.2,2.7) {$z_0$};
\node (a) at (-3.5,-2.2) {$y_0$};
\node (a) at (1.5,-2.4) {$z_1$};
\node (a) at (3.7,2.7) {$w_0$};
\node (a) at (2.5,0.8) {$y_1$};
\node (a) at (3.9,6) {$w_1$};
\end{tikzpicture}
\]
\caption{Relative strata of a submanifold with corners of the unit $3$--ball; $x$ belongs to the relative $0$--stratum; $y_0$ and $y_1$ to the relative $1$--stratum; $z_0$ and $z_1$ to the relative $2$--stratum; $w_0$ and $w_1$ to the relative $3$--stratum}
\label{fig_srata}
\end{figure}
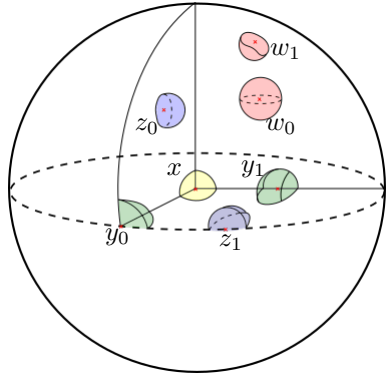
Proposition \ref{prop_smoothing} readily adapts to manifolds with boundary. Following the proof of this proposition, we define a vector field $\chi$ in $W$, which obeys the following conditions.
\begin{itemize}
\item This vector field, when restricted to $\partial W$, is tangent to $\partial W$. Seeing $Y \cap \partial W$ as a submanifold with corners of $\partial W$, $\chi_{\lvert \partial W}$ is pointing inside $Y \cap \partial W$ on $\partial (Y \cap \partial W)$, and is transverse to all the strata of $\partial(Y \cap \partial W)$.
\item On the interior of $W$, $\chi$ is pointing inside $Y$ on $\partial Y$, and is transverse to all the strata of the boundary of $Y \cap Int(W)$.   
\end{itemize} 
Such a vector field induces an ambient isotopy $F$ of $W$, smooth on each stratum of $Y$, such that~$F(Y,1)$ is a subset of $Y$, as well as a submanifold with corners of $W$ with only two non-empty relative strata. Denoting by $m$ the dimension of $Y$, and by $m(Y)$ its top-dimensional relative stratum, the $m$--dimensional relative stratum of $F(Y,1)$ is~$F(m(Y),1)$; the $(m-1)$--dimensional relative stratum of $F(Y,1)$ is $F(Y \setminus m(Y) ,1)$; all the lower dimensional relative strata are empty. We define (up to ambient isotopy) the \emph{relative smoothing} of $Y$ as $s_r(Y)=F(Y,1)$, with $F$ such an isotopy. 
\par We can smoothen the remaining corners by pushing $F(Y,1) \cap \partial W$ inside $F(Y,1)$, along a vector field pointing inwards on $F(Y,1) \cap \partial W$ (see Figure \ref{fig_smoothing}). We define, up to ambient isotopy, the \emph{smoothing} $s(Y)$ of $Y$ as the smooth submanifold thus obtained. Note that the smoothing of $Y$ is included in the relative smoothing of $Y$ and in the interior of $W$. In this article, we use both versions of smoothings.
\begin{figure}[h!]
\[
\begin{tikzpicture}[scale=0.3]
\node (mypic) at (0,0) {\includegraphics[scale=0.5]{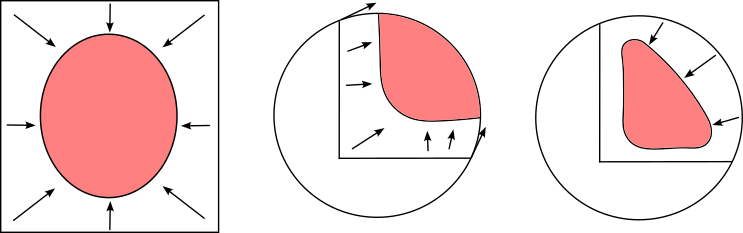}};
\end{tikzpicture}
\]
\caption{Left: smoothing of a rectangle in $\mathbb{R}^2$; Middle: relative smoothing of a submanifold with corners of a disk; Right: smoothing of the same submanifold}
\label{fig_smoothing}
\end{figure} 
\begin{Remark}
\label{rmk_smooth_product}
In this article, we also need to smoothen the corners of the product $X \times B$, with~$X$ a smooth manifold with boundary and $B$ a smooth closed ball. This can be done by considering~$B$ as a subset of a bigger open ball~$B'$. Then $X \times B$ is embedded as a submanifold with corners of the product $X \times B'$, which is a smooth manifold with boundary. Now we can smoothen the corners of~$X \times B$ considered as a submanifold of~$X \times B'$.  
\end{Remark} 
\subsection{Multisections} 
\begin{Definition}
\label{def_multisection} 
An $n$--section of a smooth, connected, oriented, closed $(n+1)$--manifold~$W$ is a decomposition $W= \bigcup_{i=1}^n W_i$, such that, for every non-empty subset $I$ of $\lbrace 1,...,n \rbrace$, the intersection $W_I = \bigcap_{i \in I} W_i$ is an $(n- \lvert I \rvert +2)$--dimensional submanifold with corners of~$W$ verifying the following conditions.
\begin{itemize}
\item For every $I$ and $\lvert I \rvert \leq k \leq n$, the $(n - k +2)$--stratum of $W_I$ is the disjoint union of the interiors of the $W_J$'s, such that $J \supseteq I$ and $\lvert J \rvert =k$. 
\item If $I= \lbrace1,...,n \rbrace$, $W_{1...n}$ is a closed smooth surface.
\item If $\lvert I \rvert \leq n-1$, the smoothing of $W_I$ is diffeomorphic to an $(n- \lvert I \rvert +2)$--dimensional handlebody.
\end{itemize}
We call the intersection $W_{1...n}$ the \emph{central surface} of the multisection, its genus the \emph{genus of the multisection}, and one $W_I$ a \emph{piece} of the multisection. The union of the $k$--dimensional pieces is the \emph{$k$--spine} of the multisection. 
\end{Definition}
A schematic for a trisection and a quadrisection is featured on Figure \ref{fig_schematic_multisection}. The boundaries of the balls that represent each manifold are dashed, to emphasize that these manifolds are closed.
\begin{figure}[h!]
\[
\begin{tikzpicture}[scale=0.3]
\node (mypic) at (0,0) {\includegraphics[scale=0.5]{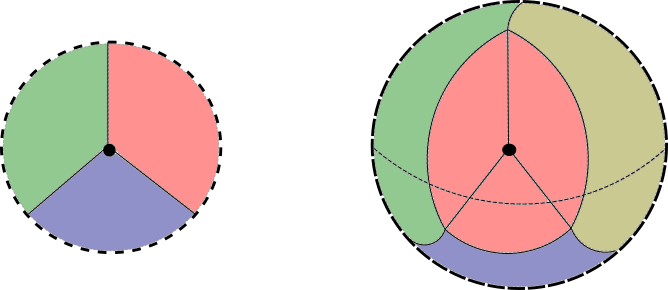}};
\node (a) at (-7,1) {$W_1$};
\node (a) at (-13,1) {$W_3$};
\node (a) at (-10,-3) {$W_2$};
\node (a) at (6,1) {$W_1$};
\node (a) at (3,1) {$W_3$};
\node (a) at (13,1) {$W_4$};
\node (a) at (8,-5.5) {$W_2$};
\end{tikzpicture}
\]
\caption{Schematic for a trisection (left) and a quadrisection (right)}
\label{fig_schematic_multisection}
\end{figure}
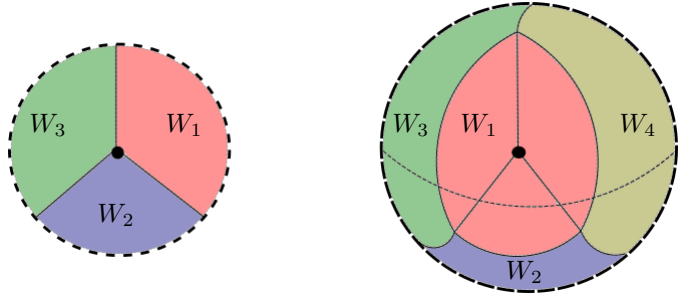
\par Let $I \subset \lbrace 1,...,n \rbrace$. Because the pieces of a multisection are submanifolds with corners, every point in the interior of $W_I$ admits a neighborhood that is decomposed according to the following \emph{local model}.
\begin{Definition}
\label{def_std_model}
Let $\Delta^{k-1}$ be the standard $(k-1)$--simplex, with vertices $\lbrace p_i \rbrace_{1 \leq i \leq k}$ and barycenter~$p_0$. Let $\Delta^{k-1} = \cup_{i=1}^k \Delta_i$ be the subdivision of $\Delta^{k-1}$, where $\Delta_i$ is the convex hull of $\lbrace p_j \mid 0 \leq j \leq k \; \mathrm{and} \; j \neq i \rbrace$. A \emph{standard $k$--decomposition} of $\mathbb{R}^{k-1}$ is a decomposition $\mathbb{R}^{k-1} = \cup_{i=1}^k A_i$, such that $\cap_{i=1}^k A_i = \lbrace 0 \rbrace$, and that there exists a permutation $\sigma$ of $\lbrace 1,...,k \rbrace$, and a diffeomorphism $\phi: \mathbb{R}^{k-1} \rightarrow Int(\Delta^{k-1})$ which sends $A_i$ to $\Delta_{\sigma(i)} \cap Int(\Delta^{k-1})$ for all $i$.   
\end{Definition} 
Then the local model around around $x \in Int({W}_I)$ is a neighborhood $U_x$ of $x$ in $W$, together with a standard $k$--decomposition $\mathbb{R}^{k-1} = \cup_{i=1}^k A_i$, and a diffeomorphism from $U_x$ to $\mathbb{R}^{n-k+2} \times \mathbb{R}^{k-1}$ sending $W_i \cap U_x$ to $\mathbb{R}^{n-k+2} \times A_i$ for all $i \in I$.
\par As an illustration, and for further use, we recall some basic examples tho\-roughly described in \cite{aribi2023multisections}.
\begin{Example}
\label{ex_sn_g0}
For $n \geq 1$, let $S^{n+1}$ be the unit sphere of $\mathbb{R}^{n+2}$ and $\pi: S^{n+1} \rightarrow \mathbb{R}^{n-1}$ the projection map forgetting the last three coordinates. Consider a standard $n$--decomposition of $\mathbb{R}^{n-1}$ and pull it back by $\pi$. This gives a genus--$0$ multisection of $S^{n+1}$. The central surface is a $2$--sphere. All the other pieces are balls of varying dimensions.
\end{Example}
\begin{Definition}
\label{ex_sn_std}
Consider the projection $\mathbb{R}^{k+1} \rightarrow \mathbb{R}^{k-1}$ forgetting the last two coordinates. The
ima\-ge of the unit sphere $S^k$ is the unit ball $B^{k-1}$, which inherits, by inclusion, a decomposition from a standard $k$--decomposition of $\mathbb{R}^{k-1}$. We call the preimage of this decomposition a \emph{standard} decomposition of $S^k$.  
\end{Definition}
Denoting by $S^k = \cup_{i=1}^k B_i$ a standard decomposition of $S^k$, we easily compute that $\cap_{i \in I} B_i$ is a $(k - \lvert I \rvert + 1)$--ball for every non-empty $I \subset  \lbrace 1, . . . , k \rbrace$ and that $\cap_{i=1}^k B_i$ is a circle. 
\begin{Example}
\label{ex_sn_g1}
For $0<k<n$, consider $S^{n+1} = \partial B^{n+2} \simeq \partial (B^{k+1} \times B^{n-k+1}) = (S^{k} \times B^{n-k+1}) \cup (B^{k+1} \times S^{n-k})$, where $\simeq$ means we have smoothened corners at $S^{k}\times S^{n-k}$ as in Remark \ref{rmk_smooth_product}. Let~$S^k = \cup_{i=1}^k A_i$ and $S^{n-k} = \cup_{i=1}^{n-k} B_i$ be standard decompositions of $S^k$ and $S^{n-k}$. Then $S^{n+1} = (\cup_{i=1}^k A_i \times B^{n-k+1}) \cup (\cup_{i=1}^{n-k} B_i \times B^{k+1})$ is a genus--$1$ multisection of $S^{n+1}$.    
\end{Example}
\begin{Example}
\label{ex_s1sn}
We obtain a genus--$1$ multisection of $S^1 \times S^n$ by taking the standard decomposition of $S^n$ and multiplying each factor by $S^1$.
\end{Example}
\section{Relative multisections and relative fibrations}
\label{sec_rel_mult}
\subsection{Relative multisections}  
\begin{Definition}
\label{def_smooth_comp_body} 
A \emph{$k$--dimensional smooth compression body} is a connected manifold $C$ cons\-tructed as follows.
\begin{itemize}
\item take the product $Y \times [0,1]$, with $Y$ a compact smooth $(k-1)$--manifold; 
\item smoothly attach $k$--dimensional $1$--handles along $Y \times \lbrace 1 \rbrace$; 
\item smoothen corners at $\partial Y \times \lbrace 0 \rbrace$ and at $\partial Y \times \lbrace 1 \rbrace$. 
\end{itemize}
The image of $Y \times \lbrace 0 \rbrace$ under the smoothing is the \emph{negative boundary} $\partial_- C$ of the compression body; the closure of its complement in $\partial C$ is the \emph{positive boundary} $\partial_+ C$ of the compression body.
\end{Definition}
\begin{figure}[h!]
\[
\begin{tikzpicture}[scale=0.3]
\node (mypic) at (0,0) {\includegraphics[scale=0.5]{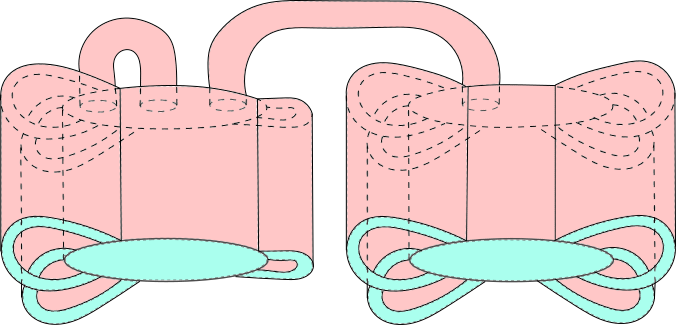}};
\end{tikzpicture}
\]
\caption{Example of a tridimensional compression body before smoothing, with its negative boundary in blue and its positive boundary in pink}
\label{fig_comp_bdy}
\end{figure}
\begin{Definition}
\label{def_p_based}
Let $P$ be a compact surface with no closed component. For $k \geq 3$, let $s(P \times B^{k-3})$ be the smoothing of the product of $P$ with the closed $(k-3)$--ball. A \emph{$P$--based} $k$--dimensional compression body is a smooth compression body $C$, where we have taken $Y = s(P \times B^{k-3})$ in Definition \ref{def_smooth_comp_body}. The image of $\partial P \times \lbrace 0 \rbrace$ under $s$ is a link, which we call the \emph{binding} of~$\partial_{-}C$. 
\end{Definition} 
\begin{Remark}
\label{rmk_dec_neg_bdy_cmp}
Because $P$ has no closed component, a $P$--based compression body $C$ is a handlebody, together with a specified decomposition of its boundary. The negative boundary of $C$ is naturally decomposed as a union of copies of $P$, properly embedded in~$\partial C$, with disjoint interiors, and intersecting two by two along the binding. Therefore, one can think of the smoothing of $\partial_- C$ as the product $P \times B^{k-3}$, with the $\partial P \times B^{k-3}$ part of its boundary collapsed to $\partial P \times \lbrace 0 \rbrace$. The positive boundary of a $P$--based compression body is diffeomorphic to a connected sum, whose summands are the components of the smoothing of $P \times B^{k-3}$ and a number of copies of $S^1 \times S^{k-2}$.   
\end{Remark}      
\begin{Definition}
\label{def_rel_multi}
A relative $n$--section of a smooth, connected, oriented, compact manifold~$W$ of dimension $n+1$, with non-empty boundary, is a decomposition $W= \bigcup_{i=1}^n W_i$, such that, for every non-empty subset $I$ of $\lbrace 1,...,n \rbrace$, the intersection $W_I = \bigcap_{i \in I} W_i$ is an $(n- \lvert I \rvert +2)$--dimensional submanifold with corners of $W$, verifying the following conditions. 
\begin{itemize}
\item For every $I$ and $\lvert I \rvert \leq k \leq n$, the relative $(n - k +2)$--stratum of $W_I$ is the disjoint union of the top-dimensional relative strata of the $W_J$'s, such that $J \supseteq I$ and $\lvert J \rvert = k$.
\item If $I = \lbrace 1,...,n \rbrace$, $\Sigma = W_{1...n}$ is a smooth, connected, compact surface with non-empty boun\-dary $\partial \Sigma = \Sigma \cap \partial W$.
\item There exists a compact surface $P$, with no closed component, such that, for all $I$ with~$\lvert I \rvert \leq n-1$:
\begin{itemize}
\item there exists a $P$--based $(n - \lvert I \rvert +2)$--dimensional compression body $C_I$, and a diffeomorphism that sends the smoothing $s_I(W_I)$ to $C_I$;
\item this diffeomorphism sends $s_I(W_I \cap \partial W)$ to $\partial_- C_I$ and $s_I(\partial \Sigma)$ to the binding of $C_I$.
\end{itemize}
\end{itemize} 
We call the intersection $W_{1...n}$ the \emph{central surface} of the multisection, its genus the \emph{genus of the multisection}, the surface $P$ the \emph{page} of the multisection, and one $W_I$ a \emph{piece} of the multisection. We call $W_I \cap \partial W$ the \emph{negative boundary} $\partial_- W_I$ of $W_I$. The closure of the complement of~$\partial_- W_I$ in~$\partial W_I$ is the \emph{positive boundary} $\partial_+ W_I$ of $W_I$. The union of the $k$--dimensional pieces is the \emph{$k$--spine} of the relative multisection. 
\end{Definition} 
A schematic for a relative multisection is featured on Figure \ref{fig_schematic_relative_multisection}.\begin{Remark}
\label{rmk_PL}
The definition of a (relative) multisection immediately adapts to the PL category; in this context, we do not specify corners, as the pieces are PL-homeomorphic to PL model pieces.
\end{Remark}
\begin{figure}[h!]
\[
\begin{tikzpicture}[scale=0.3]
\node (mypic) at (0,0) {\includegraphics[scale=0.5]{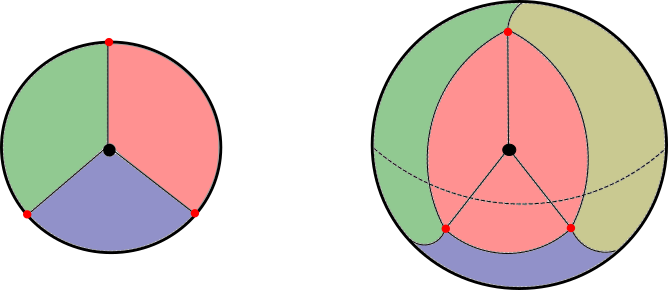}};
\node (a) at (-7,1) {$W_1$};
\node (a) at (-13,1) {$W_3$};
\node (a) at (-10,-3) {$W_2$};
\node (a) at (6,1) {$W_1$};
\node (a) at (3,1) {$W_3$};
\node (a) at (13,1) {$W_4$};
\node (a) at (8,-5.5) {$W_2$};
\end{tikzpicture}
\]
\caption{Schematic for a relative trisection (left) and a relative quadrisection (right)}
\label{fig_schematic_relative_multisection}
\end{figure}
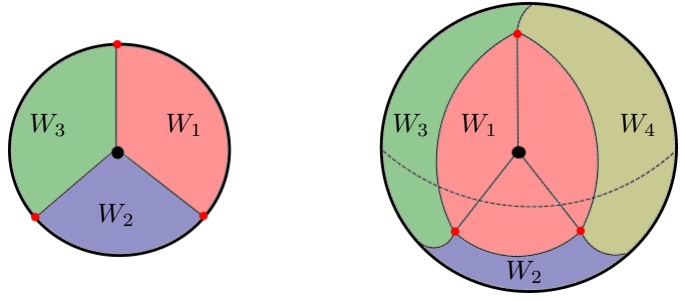 
\begin{Remark}
\label{rmk_dec_pos_bdy}
Let $\lvert I \rvert \leq n-2$. Denote by $s_r(W_I)$ the relative smoothing of $W_I$. Then the relative $n$--section of $W$ induces a relative $(n- \lvert I \rvert)$--section of $s_r(\partial_+ W_I)$, given by the union, for $J \supset I$ with $\lvert J \vert = \lvert I \rvert +1$, of the $s_r(W_J)$'s (see Figure \ref{fig_retraction}).  
\end{Remark}
\begin{figure}[h!]
\[
\begin{tikzpicture}[scale=0.3]
\node (mypic) at (0,0) {\includegraphics[scale=0.5]{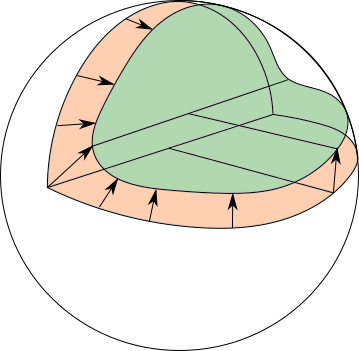}};
\end{tikzpicture}
\]
\caption{The relative multisection of $W$ induces a relative multisection of the image of~$\partial_+W_I$ by the relative smoothing of $W_I$ \\ \footnotesize{A flat angle appears in order to make the picture easier to understand}}
\label{fig_retraction}
\end{figure} 
A relative $2$--section of a compact $3$--manifold is a balanced sutured Heegaard splitting, as defined by \textcite{goda1992heegaard} (see also Juh\'asz \cite{juhasz2006holomorphic} for a concise exposition). A relative $3$--section of a compact manifold of dimension $4$ corresponds to a relative trisection, as defined in \cite{gay2016trisecting,castro2016relative,castro2018diagrams,castro2018trisections}. By specifying that the pieces of a relative multisection are submanifolds with corners, endowed with a given relative stratification, we have fixed a local model around each point (similarly as in the closed case). This local model is described below as a direct consequence of Definition \ref{def_rel_multi}.
\begin{Proposition}
\label{local_model}
Let $W= \bigcup_{i=1}^n W_i$ be a relative $n$--section of a smooth manifold $W$. Let $I$ be a subset of $\lbrace 1,...,n \rbrace$ with $\lvert I \rvert = k$. Then:
\begin{itemize}
\item for every $x$ in the intersection of the top dimensional relative stratum of $W_I$ with the interior of $W$, there is a neighborhood $U_x$ of $x$ in $W$, a standard $k$--decomposition $\mathbb{R}^{k-1} = \cup_{i=1}^k A_i$, and a diffeomorphism from $U_x$ to $\mathbb{R}^{n-k+2} \times \mathbb{R}^{k-1}$ which sends $W_i \cap U_x$ to $\mathbb{R}^{n-k+2} \times A_i$ for every $i \in I$;
\item for every $y$ in the intersection of the top dimensional relative stratum of $W_I$ with the boundary of~$W$, there is a neighborhood $U_y$ of $y$ in $W$, a standard $k$--decomposition $\mathbb{R}^{k-1} = \cup_{i=1}^{k} A_i$, and a diffeomorphism from $U_y$ to $(\mathbb{R}^{n-k+1} \times [0,\infty)) \times \mathbb{R}^{k-1}$ sending $W_i \cap U_y$ to $(\mathbb{R}^{n-k+1} \times [0, \infty)) \times A_i$ for every $i \in I$.
\end{itemize}
\end{Proposition}
\begin{Example}
\label{ex_b_n}
Let $B^{n+1} = \lbrace (x_1,...,x_{n+1}) \in \mathbb{R}^{n+1} \mid \Sigma_{i=1}^{n+1} x_i^2 \leq 1 \rbrace$ be the closed unit ball of~$\mathbb{R}^{n+1}$. Denote by $B$ the $(n-1)$--dimensional ball $\lbrace (x_1,...,x_{n-1},0,0) \in B^{n+1} \rbrace$, and by $\pi$ the projection from $B^{n+1}$ to $B$. Denote by $\mathbb{R}^{n-1} = \cup_{i=1}^n A_i$ a standard $n$--decomposition of~$\mathbb{R}^{n-1}$. Consider the decomposition $B = \cup_{i=1}^n B_i$, with $B_i = B \cap A_i$. Then the decomposition $B^{n+1} = \cup_{i=1}^n \pi^{-1} (B_i)$ is a relative $n$--section of the closed $(n+1)$--dimensional unit ball. 
\end{Example}
\begin{Example}
\label{ex_sigma_bn}
Let $W$ be the smoothing of $\Sigma \times B^{n-1}$, with $\Sigma$ a connected compact surface with non-empty boun\-dary. Consider a decomposition $B^{n-1} = \cup_{i=1}^n B_i$, induced by a standard $n$--decomposition of $\mathbb{R}^{n-1}$, as in Example \ref{ex_b_n}. Then $W \simeq \cup_{i=1}^n \Sigma \times B_i$ is a relative \mbox{$n$--section} of $W$. Note that (identifying $B^{n+1}$ with $D^2 \times B^{n-1}$) the relative multisection of Example \ref{ex_b_n} is a special case of this construction, with $\Sigma=D^2$.
\end{Example}
\begin{Example}
\label{ex_multi_i_times_sn} 
Let $W=I \times S^n$. Start from the standard decomposition $S^n = \cup_{i=1}^{n}B_i$, then set $W_i=I \times B_i$. This defines a relative $n$--section of $W$. The central surface is a cylinder, the $3$--dimensional intersections are obtained from this cylinder by gluing in a single $2$--handle. 
\end{Example}  
\begin{Example}
\label{ex_xn_minus_ball}
Let $Y=\cup_{i=1}^n Y_i$ be a multisected closed $(n+1)$--manifold. Consider the compact manifold with spherical boundary $W= Y \setminus Int(B)$, where $B$ is an $(n+1)$--ball. We can assume that~$B$ is a closed tubular neighborhood of a point belonging to the central surface of the multisection, such that $B \cap Y_I$ is an $(n- \lvert I \rvert +2)$--dimensional ball for every~$I$. Then $W = \cup_{i=1}^n \overline{Y_i \setminus B}$ is a relative $n$--section of $W$, whose central surface is the central surface of the multisection of $Y$ minus the interior of a disk. 
\end{Example}
\begin{Example}
\label{ex_connected_sum}
The connected sum of two (relatively or not) multisected manifolds is again a multisected manifold (by taking the connected sum on a neighborhood of a point on each central surface). The central surface is the connected sum of the two central surfaces. 
\end{Example}

\begin{Remark}
\label{rmk_stab}
A relative multisection of a manifold with boundary is not unique: the \emph{stabilization} moves defined in \cite{aribi2023multisections} adapt readily to the relative setting, as they can be performed inside the interior of the manifold. Stabilizing a relative multisection is taking the connected sum of this multisection with a genus--$1$ multisection of the sphere. We refer to \cite[Section 5]{aribi2023multisections} for a thorough description. 
\end{Remark}

\begin{Remark}
\label{rmk_decomp_bdy_piece}
As we mentioned in Remark \ref{rmk_dec_pos_bdy}, a relative multisection $\mathcal{M}$ induces a relative multisection of the smoothing of the positive boundary of the $k$--dimensional pieces, with $k \geq 4$. Denote by $C$ such a piece, and by $P = \sqcup_{i=1}^{\ell} P_i$ the page of $\mathcal{M}$, with each $P_i$ connected. For every~$i$, denote by $\mathcal{M}_i$ the multisection of $P_i \times B^{k-3}$ of Example \ref{ex_sigma_bn}. Then the relative multisection induced by $\mathcal{M}$ on the smoothing of~$\partial_+ C$ is the connected sum of the $\mathcal{M}_i$'s with a multisection~$\mathcal{M}'$ of connected sums of copies of $S^1 \times S^{k-2}$. For instance, $\mathcal{M}'$ could be connected sums of the multisection of $S^1 \times S^{k-2}$ described in Example \ref{ex_s1sn}, possibly with stabilizations (but we do not know if this is the only possibility). 
\end{Remark} 

\subsection{Relative fibrations}
\label{rel_fib}
As we shall see, a relative multisection induces a specific decomposition of the boundary of the multisected manifold, which we call a \emph{relative fibration}. 
\begin{Definition}
\label{def_rel_fibration}
A \emph{relative fibration} of a smooth closed $n$--manifold $Y$ is a pair $(\mathcal{B}, \pi)$ where:
\begin{itemize}
\item $\mathcal{B}$ is an oriented link in $Y$, called the binding of the relative fibration;
\item $\pi : Y \setminus \mathcal{B} \rightarrow S^{n-2}$ is a locally trivial smooth fibration, such that the closure of each fiber is a compact surface with boundary $\mathcal{B}$.  
\end{itemize}
\end{Definition}
We call the surface $P \simeq \overline{\pi^{-1}(x)}$, for every $x \in S^{n-2}$, the \emph{page} of the relative fibration. A relative fibration of a smooth surface is a \emph{balanced sutured decomposition} of this surface, i.e. a decomposition of this surface into two compact surfaces of same genus. A relative fibration of a \mbox{$3$--dimensional} smooth manifold is an open-book decomposition. Note that Definition \ref{def_rel_fibration} naturally adapts to PL manifolds. 
\begin{Proposition}
\label{prop_induced}
An $n$--section $W= \cup_{i=1}^n W_i$ induces a relative fibration on $\partial W$.
\end{Proposition} 
\begin{proof}
Let $\mathcal{B}=\Sigma \cap \partial W$, with $\Sigma$ the central surface of the multisection. Then each $X_i=(W_i \cap \partial W) \setminus \mathcal{B}$ is diffeomorphic to the product of the interior of $P$ with an $(n-2)$--ball with corners. If $i<n$, $X_i$ is attached to $\cup_{j=1}^{i-1}X_j$ along the product of the interior of $P$ with an $(n-3)$--ball with corners. Therefore, by induction, $\cup_{j=1}^{n-1}X_j$ is piecewise-diffeomorphic to the product of the interior of $P$ with an $(n-2)$--ball. Now the boundary of $\cup_{j=1}^{n-1}X_j$ is $\partial X_n$, which can be pushed and smoothened inside $X_n$ using a relative smoothing of $W_n$. This means that $\partial W \setminus \mathcal{B}$ is the smooth union of two products $B^{n-2} \times Int(P)$ along $S^{n-3} \times Int(P)$, therefore a smooth fibration over $S^{n-2}$, with fiber the interior of $P$. Finally, the boundary of the closure of each fiber in $\partial W$ is $\mathcal{B}=\Sigma \cap \partial W$, which concludes the proof. The argument adapts easily to the PL category.         
\end{proof}
\begin{Example}
\label{ex_rel_fib_sn}
For $n \geq 2$, consider the $n$--sphere $S^n = \lbrace (x_1,...,x_{n+1}) \in \mathbb{R}^{n+1}, \Sigma_{i=1}^{n+1} x_i^2 =1 \rbrace$. Denote by $\mathcal{B}$ the circle $\lbrace (0,...,0,x_n,x_{n+1}) \in S^n \rbrace$. Let $\pi: S^n \setminus \mathcal{B} \rightarrow S^{n-2}$ be the smooth projection defined by $\pi((x_1,...,x_{n+1}))= || (x_1,...,x_{n-1}) ||^{-1} (x_1,...,x_{n-1})$. Then $(\mathcal{B},\pi)$ is a relative fibration of~$S^n$ with binding a circle and page a disk. This relative fibration is induced on $\partial B^{n+1} = S^n$ by the relative $n$--section of $B^{n+1}$ described in Example \ref{ex_b_n}.
\end{Example}
\begin{Example}
\label{ex_rel_fib_s1_sn}
Fix $n \geq 3$ and consider the product $S^{n-1} \times S^1$. Choose two disjoint points~$x$ and~$y$ in $S^{n-1}$. Then $S^{n-1} \setminus \lbrace x,y \rbrace$ is diffeomorphic to $S^{n-2} \times (-1,1)$. We thus have a fibration of~$(S^{n-1} \setminus \lbrace x,y \rbrace ) \times S^1 \simeq (S^{n-2} \times (-1,1) ) \times S^1 \simeq S^{n-2} \times ( (-1,1) \times S^1)$. Therefore $S^{n-1} \times S^1$ admits a relative fibration, with page an annulus, and binding the two circles $ \lbrace x \rbrace \times S^1$ and~$\lbrace y \rbrace \times S^1$. This relative fibration is induced on $S^{n-1} \times S^{1} = \partial (S^{n-1} \times D^2)$ by the relative $n$--section of $S^{n-1} \times D^2$ described in Example \ref{ex_multi_sn1_times_d2}.
\end{Example}
\begin{Example}
\label{ex_rel_fib_bdy}
Let $P$ be a connected compact surface of genus $p$, with $b > 0$ boundary components. Let $\ell=2p+b-1$. For $n \geq 3$, we construct a relative fibration of $\#^{\ell} S^1 \times S^{n-1}$, with page~$P$. Consider $H = P \times B^{n-1}$, an $(n+1)$--dimensional handlebody (with corners) of genus $\ell$. We decompose its boundary as $\partial H = (P \times \partial B^{n-1}) \cup (\partial P \times B^{n-1}) = (P \times S^{n-2}) \cup (\partial P \times B^{n-1})$, and denote by $\mathcal{B}$ the $b$--components link $\partial P \times \lbrace 0 \rbrace \subset \partial P \times B^{n-1}$. This provides a diffeomorphism between $(\#^{\ell} S^1 \times S^{n-1}) \setminus \mathcal{B}$ and $Int(P) \times S^{n-2}$. This is the requested relative fibration, which corresponds to the decomposition induced by the relative $n$--section of~$P \times B^{n-1}$ of Example \ref{ex_sigma_bn}.
\end{Example}
It is well-known that any closed surface or $3$--manifold admits a relative fibration. However, it is easy to show that this is not the case for any $2n$--manifold $Y$, with $n \geq 2$. Using the relative fibration to compute the Euler characteristic of $Y$, we obtain $ \chi(Y) = \chi (S^{2n-2}) \chi (P)$. Therefore $\chi (Y) = 2 \chi (P)$. But, as $\chi(\mathbb{C}P^{2n} \# (\#^{n-1} T^{4n}))=3$ and $\chi(\mathbb{C}P^{2n} \# (\#^{n} T^{4n}))=1$, any integer, odd or even, is the Euler characteristic of a connected sum of copies of $\mathbb{C}P^{2n} \# (\#^{n-1} T^{4n})$ and~$\mathbb{C}P^{2n} \# (\#^{n} T^{4n})$. Restricting our attention to null-cobordant manifolds, which necessa\-rily have even Euler charac\-teristic, we still find a restriction because of the condition~$\chi(Y) \leq 2$: for instance, the Euler characteristic of $S^{2n} \times S^{2m} = \partial (B^{2n+1} \times S^{2m})$ is equal to $4$ for any positive $n$ and $m$, which shows that $S^{2n} \times S^{2m}$ cannot admit a relative fibration, and therefore that $B^{2n+1} \times S^{2m}$ cannot admit a relative multisection. We now show that, in fact, very few manifolds of dimension greater than~$3$ relatively fiber.     
\par The relative fibrations described in the previous examples were actually trivial, in the sense that the preimage of $S^{n-2}$ by the projection was diffeomorphic to the product $Int(P) \times S^{n-2}$. We will prove that this fact is general if $n \geq 4$. As we shall see, this immediately implies that the previous examples covered \emph{all} possible cases, i.e. that if $n \geq 4$, then the components of an $n$--manifold admitting a relative fibration are necessarily $S^n$ or connected sums of copies of $S^1 \times S^{n-1}$.  
\par We will use the following result, by Earle and Schatz. We denote by $\mathrm{Diff}(P, \partial)$ the group of diffeomorphisms of $P$ that are the identity on $\partial P$.
\begin{Theorem}[Earle and Schatz, \cite{earle1970teichmuller}]
\label{Earle}
Let $P$ be a smooth connected compact surface with non-empty boun\-dary. Then the components of $\mathrm{Diff}(P, \partial)$ are contractible. 
\end{Theorem} 
There is an earlier version of this theorem for closed surfaces of genus at least $2$, due to Earle and Eells \cite{earle1969fibre}. For a topological proof of both versions, we refer to \cite[Appendix~B]{hatcher2011short}. Note that this theorem remains true when replacing diffeomorphisms by PL-homeomorphisms, as it is $2$--dimensional. The results and proofs until the end of this subsection are developed in the smooth category, but also apply to the PL category. 
\begin{Theorem}
\label{prop_rel_trivial}  
Let $Y$ be a connected, smooth $n$--manifold, admitting a relative fibration with page $P$ and binding $\mathcal{B}$. If $n=3$, then $Y$ is diffeomorphic to an abstract open-book with page $P$. If $n>3$, then $Y$ is diffeomorphic to $(P \times S^{n-2}) \cup_{\partial P \times S^{n-2}} (\partial P \times B^{n-1}) = \partial (P \times B^{n-1})$, via a diffeomorphism that sends $\mathcal{B}$ to $\partial P \times \lbrace 0 \rbrace$. In particular, if~$P$ has genus $p$ and $b>0$ boundary components, then~$Y \simeq  \#^{2p+b-1} (S^1 \times S^{n-1})$, with the convention $\#^0 (S^1 \times S^{n-1})=S^n$.
\end{Theorem}
\begin{Lemma}
\label{lemma_fibration}
Given a smooth relative fibration $(\mathcal{B}, \pi)$ of $Y$, there exists a local triviali\-zation $\lbrace(U_1, \phi_1),(U_2, \phi_2) \rbrace$ of the fibration $\pi: Y \setminus \mathcal{B} \rightarrow S^{n-2}$, satisfying the following requirement. Denoting by $\mathrm{proj}_P$ the projection from $(U_1 \sqcup U_2) \times P$ to $P$, each $\mathrm{proj}_P \circ \phi_i$ extends to a smooth map $f_i$ from $\pi^{-1}(U_i) \cup \mathcal{B}$ to $P$, such that $f_1$ and~$f_2$ agree on $\mathcal{B}$. Moreover, each $f_i$ restricts to a diffeomorphism from the closure of each fiber to $P$.
\end{Lemma}
\begin{proof}
Let $(U_1,U_2)$ be an open cover of $S^{n-2}$, and for $i=1,2$, let $(U_i \times P) / \sim$ be the quotient space where $(x,q) \sim (y,q)$ for every $x$ and $y$ in $U_i$ and $q \in \partial P$. There exists a diffeomorphism~$\psi_i$ from $(U_i \times P) / \sim$ to $\pi^{-1}(U_i) \cup \mathcal{B}$, such that $\psi_i(x,Int(P))= \pi^{-1}(\lbrace x \rbrace)$ for every $x \in U_i$. Moreover, denoting by $\mathrm{proj}_i$ the projection from $U_i \times P$ to $(U_i \times P) / \sim$, we can require that $\psi_1 \circ \mathrm{proj}_1 (x,q) = \psi_2 \circ \mathrm{proj}_2(y,q)$ for every $(x,y,q) \in U_1 \times U_2 \times \partial P$. We thus obtain a smooth map $\psi_i \circ \mathrm{proj}_i$ from $U_i \times P$ to $\pi^{-1}(U_i) \cup \mathcal{B}$, whose restriction to $U_i \times Int(P)$ is a diffeomorphism between $U_i \times Int(P)$ and $\pi^{-1}(U_i)$. We denote by $\phi_i$ the inverse of this restriction. As it preserves the fibers of $\pi$, we obtain a local trivialization $\lbrace(U_1, \phi_1),(U_2, \phi_2) \rbrace$ which fulfills our requirement.
\end{proof}
\begin{proof}[Proof of Theorem \ref{prop_rel_trivial}]
Let $\lbrace(U_1, \phi_1),(U_2, \phi_2) \rbrace$ be a local trivialization of the fibration $\pi: Y \setminus \mathcal{B} \rightarrow S^{n-2}$, satisfying the requirement of Lemma \ref{lemma_fibration}, with associated maps $f_1$ and $f_2$. We can assume that $U_1 \cap U_2 \simeq S^{n-3} \times (-1,1)$. We start by constructing a tubular neighborhood $\nu (\mathcal{B)}$ of~$\mathcal{B}$ in~$Y$, such that the restriction $\tilde{\pi}$ of $\pi$ to $\overline{Y \setminus \nu (\mathcal{B}})$ is a locally trivial fibration, verifying $\tilde{\pi}^{-1}(x)= \pi^{-1}(x) \setminus Int(\pi^{-1}(x) \cap \nu (\mathcal{B})) \simeq P$ for every $x \in S^{n-2}$. Moreover, we want $\phi_1$ to agree with $\phi_2$ on $\nu (\mathcal{B}) \cap \pi^{-1}(U_1 \cap U_2)$. Let $C \simeq \partial P \times [0,1]$ be a collar of $\partial P$ in $P$, with $\partial P$ identified with $\partial P \times \lbrace 0 \rbrace$. Let $x \in U_1 \cap U_2$. Then each $f_i$ restricts to a diffeomorphism $\tilde{f}_i$ from $\overline{\pi^{-1}(x)}$ to $P$ (see Lemma \ref{lemma_fibration}). Up to modifying $\phi_2$ by isotopy preserving the fibers, we can consider that the inverses of $\tilde{f}_1$ and~$\tilde{f}_2$ agree on $C$. Therefore, for every $q \in \partial P$, $f_{1}^{-1}(\lbrace q \rbrace \times [0,1]) \cup f_{2}^{-1}(\lbrace q \rbrace \times [0,1])$ is an \mbox{$(n-1)$--ball} centered at $f_{1}^{-1}(q) = f_{2}^{-1}(q) $. This describes $f_{1}^{-1}(C) \cup f_{2}^{-1}(C)$ as a trivial bundle over $\mathcal{B}$, with fiber an \mbox{$(n-1)$--ball}: therefore it is a trivial tubular neighborhood of $\mathcal{B} = f_{1}^{-1}(\partial P) = f_{2}^{-1}(\partial P)$ in $Y$, fulfilling our first condition by construction. Moreover, $\phi_1$ agrees with $\phi_2$ on $\nu \mathcal{B} \cap \pi^{-1}(U_1 \cap U_2)$, as requested.    
\par We thus obtain a locally trivial fibration $\tilde{\pi}: \overline{Y \setminus \nu(\mathcal{B})} \rightarrow S^{n-2}$, with fiber a compact surface with boundary diffeomorphic to $P$. By construction, this fibration admits a local trivialization given by $\lbrace(U_1, \phi_1),(U_2, \phi_2) \rbrace$, where each $\phi_i$ is restricted to $\overline{Y \setminus \nu(\mathcal{B})}$. Therefore, $\overline{Y \setminus \nu(\mathcal{B})}$ is determined, up to diffeomorphism, by a smooth map from $S^{n-3} \times (-1,1)$ to $\mathrm{Diff}(P,\partial)$, as $\phi_1$ agrees with~$\phi_2$ on $\nu (\mathcal{B}) \cap \pi^{-1}(U_1 \cap U_2)$. In fact, the diffeomorphism type of the fibration is even determined by the isotopy class of the restriction of this map to the equator, i.e. by the isotopy class of a smooth map $F: S^{n-3}\rightarrow \mathrm{Diff}(P, \partial)$ (this characterization of the fibration is referred to as a \emph{clutching construction}). Note that we can always assume that $F(x)=\mathrm{Id}$ for some $x \in S^{n-3}$. Thus we have two situations. If~$n=3$, the total space of the fibration is the mapping torus $\Sigma_f$ of $f$, with $f$ a diffeomorphism of~$P$ which restricts to the identity on $\partial P$. Then $Y$ is diffeomorphic to the union of $\Sigma_f$ with solid tori, trivially glued to each boundary component: it is a tridimensional \emph{abstract open-book} (see, for instance, \cite{etnyre2006lectures}). If $n>3$, $S^{n-3}$ is connected. Then $F(S^{n-3})$ is included into the identity component of $\mathrm{Diff}(P,\partial)$, which is contractible, by Theorem \ref{Earle}. Therefore, the fibration is trivial: $Y$ is diffeomorphic to the union of $P \times S^{n-2}$ with copies of $S^1 \times B^{n-1}$, trivially glued to each boundary component. 
\end{proof}
Therefore, if $n \geq 4$, the data of a page $P$ with genus $p$ and $b$ boundary components determines a relatively fibered manifold $X_P \simeq (P \times S^{n-2}) \cup (\partial P \times B^{n-1}) \simeq \#^{2p+b-1}S^1 \times S^{n-1}$, up to diffeomorphism. 
\begin{Remark}
\textcite{ghanwat2025murasugi} define a $k$--open-book decomposition of an $n$--manifold as a fibration of the complement of an $(n-k-1)$--submanifold over $S^k$. Our relative fibration corresponds to the case $k=n-2$.
\end{Remark}
\begin{Remark}
\label{coro_fibration}
Let $n \geq 4$. The page $\Sigma_{p,b}$ of a relative fibration of $\#^{\ell}(S^1 \times S^{n-1})$ necessarily verifies $2p+b-1=\ell$, with $b>0$. Therefore, $S^n$ admits only one relative fibration: the relative fibration with page a disk, described in Example \ref{ex_rel_fib_sn}. All possible relative fibrations for $\#^{\ell}(S^1 \times S^{n-1})$, with $\ell >0$, are described in Example \ref{ex_rel_fib_s1_sn} (in particular, $S^1 \times S^{n-1}$ admits only one relative fibration, with page an annulus).
\end{Remark}
\subsection{Murasugi sum of relative fibrations}
\label{subsec_murasugi}
Let $P$ and $P'$ be two compact surfaces with boundary, together with two properly embedded arcs~$a \subset P$ and $b \subset P'$. Let $A \simeq a \times [0,1]$ and $B \simeq b \times [0,1]$ be two tubular neighborhoods of these arcs in $P$ and $P'$. We denote by $P \ast_{a,b} P'$ (or $P \ast P'$ when no confusion arises) the Murasugi sum of~$P$ and $P'$ over $a$ and $b$, i.e. the result of gluing $P$ and $P'$ along $A$ and $B$, with~$\partial a \times [0,1]$ identified to $b \times \lbrace 0,1 \rbrace$ and $\partial b \times [0,1]$ identified to $a \times \lbrace 0,1 \rbrace$. Recall that, if~$n=3$, the Murasugi sum of the abstract open-books $(P, \phi)$ and $(P', \phi')$ on $P \ast_{a,b} P'$ is the abstract open-book $(P \ast_{a,b} P', \phi \circ \phi')$. By \textcite{gabai1983murasugi} (see \cite[Theorem 2.17]{etnyre2006lectures} for a sketch of the proof),~$(P \ast P', \phi \circ \phi')$ is diffeomorphic to the connected sum of $(P, \phi)$ and $(P', \phi')$. If $n \geq 4$, consider two $n$--dimensional relative fibrations, with respective page $P$ and $P'$. We define the Murasugi sum of these relative fibrations on $P \ast_{a,b} P'$ as the relative fibration with page $P \ast_{a,b} P'$. A Murasugi sum of two relative fibrations is diffeomorphic to their connected sum. This can be seen as a straightforward adaptation of Gabai's tridimensional proof, or as the special case $k=n-2$ of \cite[Theorem~3.2]{ghanwat2025murasugi}. An easy application shows that any relative fibration of $\#^{\ell} S^{1} \times S^{n-1}$ can be obtained by performing~$\ell$ successive Murasugi sums of fibrations with annular pages.
\section{Diagrams}
\label{section_diagram}
\subsection{Multisection diagrams}
\label{subsec_diagram}
We briefly recall the notion of multisection diagrams from \cite{aribi2023multisections}. A \emph{cut system} for a closed surface $\Sigma$ is a collection of disjoint, homologically independent simple closed curves on~$\Sigma$, such that performing surgery on $\Sigma$ along these curves results in $S^2$. A cut system \emph{defining} a $3$--dimensional handlebody $H$ is a cut system for $\Sigma \simeq \partial H$, such that each curve bounds a properly embedded disk in $H$. Then $H$ is the result of gluing $2$--handles to a collar of $\Sigma$ along a defining cut system, and capping off the resulting spherical boundary component with a $3$--handle. A \emph{multisection diagram} associated to a closed multisected $(n+1)$--manifold $W= \cup_{i=1}^n W_i$ is a tuple $\lbrace \Sigma = W_{1...n}, (\alpha^i)_{1 \leq i \leq n} \rbrace$, such that for all $i$, $\alpha^i$ is a cut system on $\Sigma$ defining the $3$--dimensional handlebody $\cap_{k \neq i} W_k$. Such diagrams are unique up to diffeomorphism of $\Sigma$, isotopy of curves, and handleslides performed independently within each family $\alpha^i$. Multisection diagrams encompass Heegaard diagrams and trisection diagrams.
\begin{Example}
\label{ex_multi_diagram_sn}
A genus--$0$ multisection diagram for $S^n$ is just a $2$--sphere with an empty collection of curves. It corresponds to the decomposition of Example \ref{ex_sn_g0}. A genus--$1$ multisection diagram for $S^n$ is a torus with a collection $\alpha$ of $k$ parallel curves (with $n > k > 0$) and a collection $\beta$ of~$n-k$ parallel curves, such that every curve in $\alpha$ intersects exactly once each $\beta$ curve. Such a diagram corresponds to one of the multisections of Example \ref{ex_sn_g1}. See Figure \ref{fig_diag_sn_s1_sn}.
\end{Example} 
\begin{Example}
\label{ex_multi_diagram_s1sn}
A torus with a collection of $n$ parallel curves is a multisection diagram for $S^1 \times S^n$. It corresponds to the multisection of Example \ref{ex_s1sn}. See Figure \ref{fig_diag_sn_s1_sn}.
\end{Example}
\begin{figure}[h!]
\[
\begin{tikzpicture}[scale=0.3]
\node (mypic) at (0,0) {\includegraphics[scale=0.5]{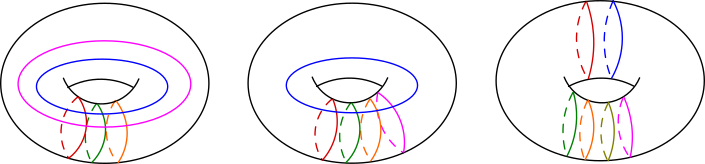}};
\end{tikzpicture}
\]
\caption{Multisection diagrams for $S^6$ (left and center) and $S^1 \times S^6$ (right)}
\label{fig_diag_sn_s1_sn}
\end{figure} 
\begin{Definition}
\label{def_diagram}
An \emph{abstract} $n$--section diagram is a tuple $\lbrace \Sigma, (\alpha^i)_{1 \leq i \leq n} \rbrace$, with $\Sigma$ a connected closed surface and each $\alpha^i$ a cut system for $\Sigma$, such that, for every $I \subset \lbrace 1,...,n \rbrace$ with $\lvert I \rvert =k \geq 2$, $\lbrace \Sigma, (\alpha^i)_{i \in I} \rbrace$ is a $k$--section diagram for a multisection of connected sums of copies of $S^1 \times S^k$. 
\end{Definition}
\begin{Theorem}[\cite{aribi2023multisections}, Theorem 3.2]
\label{thm_diagram}
Let $n \geq 2$.
\begin{itemize}
\item For $n \leq 6$, every abstract $n$--section diagram is the diagram of some smooth multisected $(n+1)$--manifold. 
\item For $n \leq 5$, this manifold is unique up to multisection-preserving diffeomorphism.
\item For every $n$, every abstract $n$--section diagram is the diagram of some PL multisected mani\-fold of dimension $n+1$, and this manifold is unique up to multisection-preserving PL-homeomorphism.  
\end{itemize}
\end{Theorem}
\begin{Remark}
\label{rmk_glue_hby}
The collections of curves in an abstract $n$--section diagram specifies how the $3$--dimen\-sional pieces are glued to the central surface. The content of Theorem \ref{thm_diagram} is that this information is actually enough to define a unique multisected manifold (with the restrictions mentioned in the theorem). For let $H$ be an $m$--dimensional handlebody. Then:
\begin{itemize}
\item if $m=3$, every diffeomorphism (resp. PL-homeomorphism) of $\partial H$ that preserves a cut system for $\partial H$ extends to a diffeomorphism (resp. PL-homeomorphism) of $H$; 
\item if $4 \leq m \leq 6$, every diffeomorphism of $\partial H$ extends to a diffeomorphism of $H$;
\item for all $m \geq 4$, every PL-homeomorphism of $\partial H$ extends to a PL-homeomorphism of~$H$.
\end{itemize}
In dimension $3$, the result holds because any diffeomorphism or PL-homeomorphism of the \mbox{$2$--sphere} extends to the $3$--ball. For $m=4$, it was proven by Montesinos \cite{montesinos1979heegaard} in the PL-category and by Laudenbach and Po\'enaru \cite{laudenbach1972note} in the smooth category (see \cite{moussard2024diffeomorphisms, meier2025equivariant} for recent proofs). It was proven by \textcite{cavicchioli1993determination} for $m=5$ and $m=6$ in the smooth category and~$m \geq 5$ in the PL-category. This result implies that there is a unique way to glue an \mbox{$m$--dimensional} handlebody to a manifold whose boundary is a connected sum of copies of~$S^1 \times S^{m-2}$. As long as it stands (indefinitely in the PL-category and until dimension $6$ in the smooth category), we can build inductively a unique manifold from a multisection diagram (see the proof of Theorem 3.2 in \cite{aribi2023multisections} for a detailed and rigorous exposition). 
\end{Remark}
\subsection{Relative multisection diagrams}
\label{subsec_rel_diagrams}
A \emph{defining cut system} for a $3$--dimensional compression body $C$ is a collection of disjoint, homologically independent simple closed curves on $\partial_+ C$, bounding properly embedded disks, such that $C$ is the result of attaching $2$--handles to a collar of $\partial_+ C$ along these curves. Then $\partial_- C$ is diffeomorphic to $\partial_+ C$ surgered along a defining cut system for $C$. A \emph{relative multisection diagram} associated to a compact multisected $(n+1)$--manifold $W= \cup_{i=1}^n W_i$ is a tuple $\lbrace \Sigma = W_{1...n}, (\alpha^i)_{1 \leq i \leq n} \rbrace$, such that for all $i$, $\alpha^i \subset \Sigma$ is a defining cut system for the $3$--dimensional compression body $\cap_{k \neq i} W_k$. 
\begin{Example}
\label{ex_rel_mult_diag_sigma}
Let $\Sigma$ be a connected compact surface with non-empty boundary. A relative multisection diagram for $\Sigma \times B^{n-1}$, associated to the relative multisection of Example \ref{ex_sigma_bn}, is just~$\Sigma$ with an empty collection of curves. 
\end{Example}  
\begin{Example}
\label{ex_diag_s1_sn}
We build a relative multisection diagram for $I \times S^n$, corresponding to the relative multisection of Example \ref{ex_multi_i_times_sn}. Keeping the notations of this example, the central surface is an annulus $\Sigma = I \times B_{1...n}$. For every $J$ with $\lvert J \rvert = n-1$, $B_J$ is a $2$--disk with boundary  $B_{1...n}$. Therefore $W_J = I \times B_J$ is obtained from $\Sigma$ by attaching a single $2$--handle along a curve parallel to $\lbrace 0 \rbrace \times B_{1...n}$. Therefore, a relative multisection diagram associated to this multisection of $I \times S^n$ is an annulus with a collection of $n$ curves parallel to its boundary.   
\end{Example} 
\begin{Example}
\label{ex_diag_xn_minus_ball}
Let $W$ be a multisected $(n+1)$--manifold, with multisection diagram $\lbrace \Sigma, (\alpha^i)_{1 \leq i \leq n} \rbrace$. Then a relative multisection diagram for $W$ minus the interior of a ball, associated to the relative multisection described in Example \ref{ex_xn_minus_ball}, is $\lbrace \overline{\Sigma \setminus D}, (\alpha^i)_{1 \leq i \leq n} \rbrace$, where $D \subset \Sigma$ is a disk disjoint from the curves in $(\alpha^i)_{1 \leq i \leq n}$. 
\end{Example}
\begin{Example}
\label{ex_diag_stab} 
Let $(W, \mathcal{W})$ be a (relatively or not) multisected manifold, with a given associated multisection diagram. Then the connected sum of this diagram with one of the genus--$1$ diagrams of Example \ref{ex_multi_diagram_sn} is a multisection diagram for a stabilization of $(W, \mathcal{W})$ (see \cite{aribi2023multisections}). 
\end{Example} 
\begin{Example}
\label{ex_connected_sum_diagrams}
Let $(W, \mathcal{W})$ and $(W', \mathcal{W'})$ be (relatively or not) multisected manifolds, with given associated multisection diagrams $\lbrace \Sigma, (\alpha^i)_{1 \leq i \leq n} \rbrace$ and $\lbrace \Sigma', (\beta^i)_{1 \leq i \leq n} \rbrace$. Then a diagram correspon\-ding to the connected sum of $(W, \mathcal{W})$ and $(W', \mathcal{W'})$ is $\lbrace \Sigma \# \Sigma ', (\alpha^i,\beta^i)_{1 \leq i \leq n} \rbrace$.
\par As an application, let $C$ be a $P$--based $k$--dimensional compression body, considered as a piece of a relative multisection $\mathcal{M}$. We describe the general structure of a relative multisection diagram associated to the relative multisection of $\partial_+ C$ induced by $\mathcal{M}$ (see Remark~\ref{rmk_decomp_bdy_piece}). Let $\lbrace P_j \rbrace_{j=1}^{\ell}$ be the connected components of $P$. Then a relative multisection diagram for this relative multisection of $\partial_+C$ is given by $ \lbrace \Sigma \# (\#_{j=1}^{\ell} P_j), (\alpha^i,\gamma^i)_{1 \leq i \leq k-2} \rbrace$, where $\lbrace \Sigma, (\alpha^i)_{1 \leq i \leq k-2} \rbrace$ is a diagram associated to a multisection $\mathcal{M}'$ of connected sums of copies of $S^1 \times S^{k-2}$, and each $\gamma^i$ is a collection of $\ell -1$ curves separating $\ell -1$ components of $P$ from $\Sigma \# (\#_{j=1}^{\ell}P_j)$. For instance, if $\mathcal{M}'$ is a connected sum of multisections of Example~\ref{ex_s1sn}, possibly stabilized, we can take for $\lbrace \Sigma, (\alpha^i)_{1 \leq i \leq k-2} \rbrace$ a connected sum of diagrams of~$S^1 \times S^{k-2}$ as in Example \ref{ex_multi_diagram_s1sn}, possibly stabilized. See Figure \ref{fig_gen_rel_diag}. 
\end{Example}
\begin{figure}[h!]
\[
\begin{tikzpicture}[scale=0.3]
\node (mypic) at (0,0) {\includegraphics[scale=0.5]{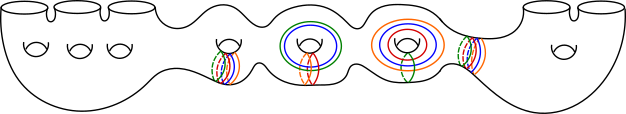}};
\end{tikzpicture}
\]
\caption{Relative $4$--section diagram for $(P \times B^{3})\# (S^1 \times S^4)$, with $P \simeq \Sigma_{3,3} \sqcup \Sigma_{1,2}$}
\label{fig_gen_rel_diag}
\end{figure}
\begin{Definition}
\label{def_rel_diagram}
Let $P$ be a compact surface with no closed component. An \emph{abstract} relative $n$--section diagram with page $P$ is a tuple $\lbrace \Sigma, (\alpha^i)_{1 \leq i \leq n} \rbrace$, with $\Sigma$ a connected compact surface with non-empty boundary, and each $\alpha^i$ a cut system for $\Sigma$, such that, for every~$I \subset \lbrace 1,...,n \rbrace$ with $\lvert I \rvert =k \geq 2$, $\lbrace \Sigma, (\alpha^i)_{i \in I} \rbrace$ is a relative $k$--section diagram associated to a relative $k$--section of the positive boundary of a $(k+2)$--dimensional $P$--based compression body. 
\end{Definition}
\begin{Remark}
\label{rmk_standard_rel_diag}
By definition, every subdiagram $\lbrace \Sigma, (\alpha^i, \alpha^j) \rbrace$ of a relative multisection diagram is a sutured Heegaard diagram for the positive boundary of a $4$--dimensional $P$--based compression body. By \cite[Theorem 3.17]{moussard2024diffeomorphisms}, such a diagram is handleslide diffeomorphic to a specific diagram, described on Figure \ref{fig_double_cmp}. 
\end{Remark}
\begin{figure}[h!]
\[
\begin{tikzpicture}[scale=0.3]
\node (mypic) at (0,0) {\includegraphics[scale=0.3]{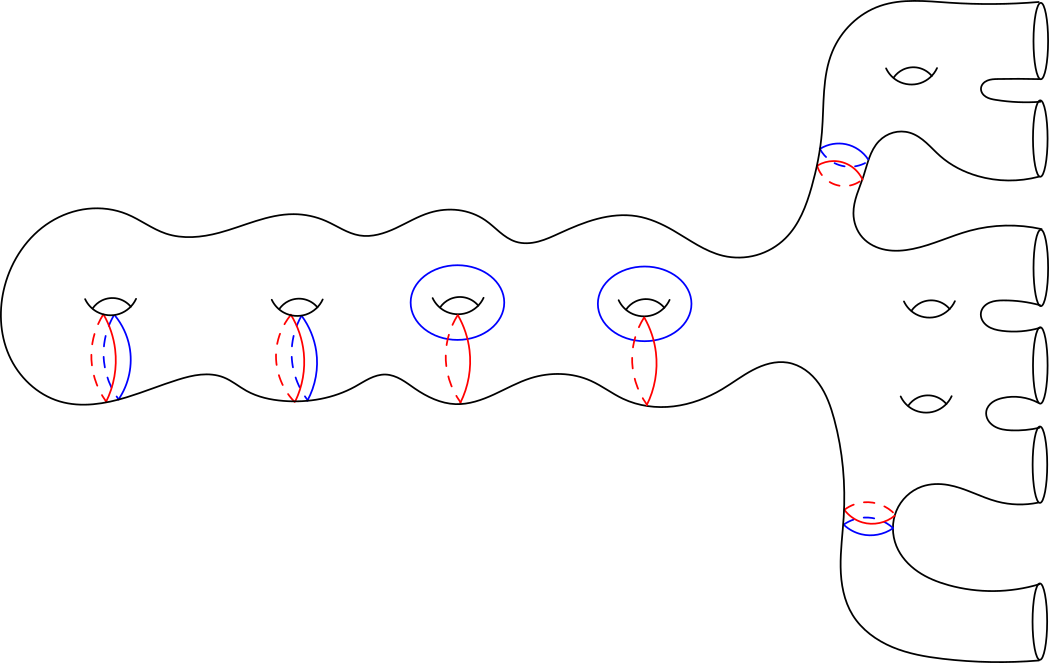}};
\end{tikzpicture}
\]
\caption{Example of a sutured Heegaard diagram for the positive boundary of a $4$--dimensional $P$--based compression body}
\label{fig_double_cmp}
\end{figure}
\begin{Remark}
\label{rmk_milnor}
As shown in \cite[Remark 3.3]{aribi2023multisections}, the existence of exotic spheres in dimension~$7$ implies that an abstract multisection diagram cannot determine a unique multisected smooth \mbox{$7$--manifold} up to diffeomorphism, because an exotic $7$--sphere $\Sigma^7$ and the $7$--sphere $S^7$ share the same multisection diagram ($S^2$ with an empty set of curves). Therefore any smooth (possibly relatively) multisected $7$--manifold $W$ shares a common diagram with $W \# \Sigma^7$.
\end{Remark}
We want to prove a relative version of Theorem \ref{thm_diagram}. To that end, we need to adapt to $P$--based compression bodies the statement that every diffeomorphism of the boundary of a $k$--dimensional handlebody extends to the whole handlebody if $4 \leq k \leq 6$, and that every PL-homeomorphism of the boundary of a $k$--dimensional handlebody extends to the whole handlebody if $k \geq 4$ (see Subsection \ref{subsec_diagram}). 
\begin{Definition}
\label{def_compat_pos_bdy}
Let $C$ be a $P$--based $k$--dimensional compression body. Identify $\partial_+C$ with the connected sum of the components of $B^{k-3} \times P$ with a number of copies of $S^1 \times S^{k-2}$. We denote by $\mathrm{Diff}(P,\partial)$ (resp. $\mathrm{PLHom}(P,\partial)$) the group of diffeomorphisms (resp. PL-homeomorphisms) of~$P$ which restrict to the identity on~$\partial P$. We say that a diffeomorphism (resp. PL-homeomorphism)~$f$ of~$\partial_+C$ is \emph{compatible with the $P$--based structure} on $C$ if, for every $(x,p) \in (S^{k-4} \times P) \cup (B^{k-3} \times \partial P) \simeq \partial (\partial_+C)$, $f(x,p) = (x,f_x(p))$, with $f_x \in \mathrm{Diff}(P,\partial)$ (resp. $f_x \in \mathrm{PLHom}(P,\partial)$).  
\end{Definition}
\begin{Proposition}
\label{prop_exten_bdy}
Let $C$ be a $P$--based $k$--dimensional compression body, with $k \geq 4$. Then every diffeomorphism (resp. PL-homeomorphism) of $\partial_+ C$, compatible with the $P$--based structure on $C$, extends to a diffeomorphism (resp. PL-homeomorphism) of $\partial C$.  
\end{Proposition}
\begin{proof}
Denote by $f$ such a diffeomorphism. 
\par First assume that $k=4$. We restate the proof of \cite[Lemma 13]{castro2018diagrams}. Identify $\partial_+C$ with the connected sum of $[-1,1] \times P$ with a number of copies of $S^1 \times S^{2}$. Consider a properly embedded arc $a \in  P$. Then $\gamma_a = (\lbrace -1 \rbrace \times a) \cup ([-1, 1] \times \partial a) \cup (\lbrace 1 \rbrace  \times a)$ is a closed curve in $\partial (\partial_+C)$, which bounds the disk $[-1,1] \times a$ in $\partial_+C$. Denote by $f_{\pm}$ the restriction of $f$ to $\lbrace \pm 1 \rbrace \times P$. Then $f(\gamma_a)= (\lbrace -1 \rbrace \times f_-(a)) \cup ([-1, 1] \times \partial a) \cup (\lbrace 1 \rbrace  \times f_+(a))$. As $f(\gamma_a)$ bounds a disk, it is homotopically trivial in $[-1,1] \times P$. Therefore $f_-(a)$ and $f_+(a)$ are homotopic rel. endpoints in $P$, and thus by a result of Baer (\cite{baer1928isotopie}, see \cite[Theorem~3.1]{epstein1966curves}) $f_-(a)$ and $f_+(a)$ are isotopic. Applying this result to a collection of arcs cutting $P$ into disks, we obtain that $f_-$ and $f_+$ are isotopic relatively to $\partial P$, through an isotopy $F: [-1,1] \times \mathrm{Diff}(P,\partial) \rightarrow \mathrm{Diff}(P,\partial)$. Therefore we can extend $f$ to~$\partial_- C$, parametrized by $[-1,1] \times P$, by setting $f(t,p)= F(t,f_-(p))$.    
\par Now assume that $k \geq 5$. By definition, for every $(x,p) \in S^{k-4} \times P$, we have $f(x,p) = (x, f_x(p))$, with $f_x \in \mathrm{Diff}(P,\partial)$. Therefore $x \rightarrow f_x$ is a smooth map from $S^{k-4}$ to $\mathrm{Diff}(P,\partial)$. As the components of $\mathrm{Diff}(P,\partial)$ are contractible (Earle and Schatz, see Theorem \ref{Earle}), and because~$k \geq 5$, this map is isotopic to a constant map $x \rightarrow g$, with $g$ fixed in $\mathrm{Diff}(P,\partial)$. Using this isotopy, we extend $f$ to a collar of $\partial(\partial_+C)$ in $\partial_-C$. Denote by $\tilde{f}$ this extension. We have $\tilde{f} = f$ on $\partial(\partial_+C)$, and $\tilde{f}(x,p)=(x,g(p))$ on the other side of the collar. Therefore, we can further extend $\tilde{f}$ to $\partial_- C$ by $(x,p) \rightarrow (x, g(p))$. The proof readily adapts to the PL category.   
\end{proof}
\begin{Proposition}
\label{prop_extension}
Let $C$ be a $P$--based $k$--dimensional compression body. If $4 \leq k \leq 6$, then every diffeomorphism of $\partial_+ C$, compatible with the $P$--based structure on $C$, extends to a diffeomorphism of $C$. If $k \geq 4$, every PL-homeomorphism of $\partial_+ C$, compatible with the $P$--based structure on $C$, extends to a PL-homeomorphism of $C$.    
\end{Proposition}
\begin{proof}
As we already mentioned, a diffeomorphism of the boundary of a smooth handlebody extends to a diffeomorphism of the whole handlebody in dimension $4$ \parencite{laudenbach1972note}, $5$ and $6$ \parencite{cavicchioli1993determination}. A PL-homeomorphism of the boundary of a PL handlebody extends to a PL-homeomorphism of the whole handlebody in any dimension $k \geq 4$ \parencite{montesinos1979heegaard,cavicchioli1993determination}. Noting that $C$ is a handlebody, Proposition~\ref{prop_extension} is a direct consequence of Proposition \ref{prop_exten_bdy}, combined with these results.
\end{proof}
\begin{Theorem}
\label{thm_rel_diagram}
Let $n \geq 2$.
\begin{itemize}
\item For $n \leq 6$, every abstract relative $n$--section diagram is the diagram of some smooth relatively multisected $(n+1)$--manifold. 
\item For $n \leq 5$, this manifold is unique up to multisection-preserving diffeomorphism.
\item For every $n \geq 2$, every abstract relative $n$--section diagram is the diagram of some PL relatively multisected $(n+1)$--manifold, and this manifold is unique up to multisection-preserving PL-homeomorphism.  
\end{itemize}
\end{Theorem}
\begin{proof}
We construct recursively a manifold corresponding to an abstract relative $n$--section diagram $\lbrace \Sigma, (\alpha^1,...,\alpha^n)\rbrace$, with page $P$. 
\par We take the product $\Sigma \times B$, with $B$ the closed $(n-1)$--ball. We equip $B$ with the decomposition $B=\cup_{i=1}^n B_i$ induced by the standard $n$--decomposition of $\mathbb{R}^{n-1}$. We denote by $p_i$ the intersection of the ray $\cap_{j \neq i} B_i$ with the boundary of $B$, and by $V_i$ a tubular neighborhood of $p_i$ in $\partial B$. For every $i$ and $j$, the disk $B_{\lbrace 1,...,n \rbrace \setminus \lbrace i ,j \rbrace}$ intersects $\partial B$ along an arc $[p_i,p_j]$, which connects $p_i$ to $p_j$. This arc intersects~$\partial V_i$ (resp.~$\partial V_j$) in a point $p_{ij}$ (resp. $p_{ji}$). See Figure \ref{fig_alg_diag_1}. Note that each $V_i$ is equipped with the decomposition induced by the decomposition of $B$, and that this decomposition is induced by the standard $(n-1)$--decomposition of $\mathbb{R}^{n-2}$. This allows us to consider that we can smoothen corners as in Remark \ref{rmk_dec_pos_bdy}, when needed. 
\par Let $C$ be a tridimensional compression body, with $\partial_+ C \simeq \Sigma$ and $\partial_-C \simeq P$. For every $i$, choose a diffeomorphism $f_i: \partial_+C \rightarrow \Sigma$, such that $f_{i}^{-1}(\alpha^i)$ is a cut system defining $C$. Then glue $C \times V_i$ along $\Sigma \times V_i$ according to $f_i \times \mathrm{Id}$. For each $i \neq j$, we have a compact $3$--manifold equipped with a sutured Heegaard splitting: $M_{ij}= (C_i \times \lbrace p_{ij} \rbrace) \cup (\Sigma \times [p_{ij},p_{ji}]) \cup (C_j \times \lbrace p_{ji} \rbrace)$. By construction, the associated sutured Heegaard diagram is $\lbrace \Sigma, (\alpha^i,\alpha^j) \rbrace$. Because a sutured Heegaard diagram defines a unique smooth manifold up to diffeomorphism, $M_{ij}$ is diffeomorphic to the positive boundary of a $P$--based $4$--dimensional compression body~$X$. Therefore, we can glue a copy of $X$ along $M_{ij}$ (and perform a similar construction for every couple $(i,j)$). We continue this process as long as the uniqueness result holds. We obtain, at each step, relatively multisected manifolds, and the constraint on every subcollection of curves ensures that each of these manifolds admits as relative multisection diagram a diagram for the positive boundary of a $P$--based compression body~$Z$. By uniqueness, we conclude that such a manifold is in fact diffeomorphic to $\partial_+Z$. Therefore we can glue $Z$ along this manifold and move one step further. Note that this argument applies also to the PL category. 
\par Now we prove the uniqueness part. Let $W$ and $W'$ be two relatively multisected PL or smooth manifolds, sharing the associated diagram $\lbrace \Sigma, (\alpha^1,...,\alpha^n)\rbrace$. Therefore we have a diffeomorphism, or PL-homeomorphism $h$ between the central surfaces of $W$ and $W'$, which preserves every cut system~$\alpha^i$. We can thus extend $h$ to the $3$--spine of $W$ and $W'$, which is the boundary of the \mbox{$4$--dimensional} pieces. Then, using Proposition \ref{prop_extension}, we can extend $h$ step by step until the $6$--spine in the smooth category, and without restriction in the PL-category (by construction, the extension of $h$ constructed at each step is compatible with the $P$--based structure on every piece).                        
\end{proof}
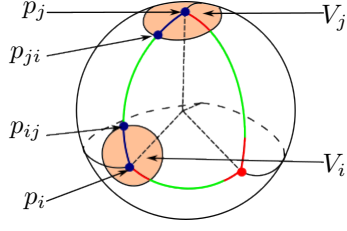
\begin{figure}[h!]
\[
\begin{tikzpicture}[scale=0.3]
\node (mypic) at (0,0) {\includegraphics[scale=0.5]{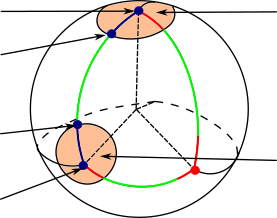}};
\node (mypic) at (-6.6,-4) {$p_i$};
\node (mypic) at (-6.6,4.4) {$p_j$};
\node (mypic) at (-7,-0.9) {$p_{ij}$};
\node (mypic) at (-7,2.4) {$p_{ji}$};
\node (mypic) at (6.6,4.1) {$V_{j}$};
\node (mypic) at (6.6,-2.4) {$V_{i}$};
\end{tikzpicture}
\]
\caption{Arcs $[p_i,p_{ij}]$, $[p_{ij},p_{ji}]$ and $[p_j,p_{ji}]$ on $B^3$}
\label{fig_alg_diag_1}
\end{figure} 
\section{Combining relative multisections}
\label{section_gluing}
\subsection{Gluing relative multisections}
Let $W$ and $W'$ be two smooth manifolds with boundary, such that there exists a diffeomorphism $f: \partial W \rightarrow \partial W'$. Extend $f$ to a diffeomorphism from a collar of $\partial W$ to a collar of $\partial W'$, and glue $W$ to $W'$ along these collars with respect to this extension of $f$. We denote by $W \cup_f W'$ the resulting smooth manifold. 
\begin{Theorem}
\label{thm:gluing}
Let $(W,\mathcal{W}=\cup_{i=1}^n W_i)$ and $(W',\mathcal{W}'=\cup_{i=1}^n W_i')$ be two relatively \mbox{$n$--sected} smooth $(n+1)$--manifolds. Suppose that there exists an orientation reversing diffeomorphism $f: \partial W \rightarrow \partial W'$, sending $\partial W_i$ to $\partial W_{i}'$ for every $i$, which preserves the fibers of the induced relative fibrations on the boundaries, i.e. such that the following diagram is commutative.

\[
  \begin{tikzcd}
    \partial W \setminus \mathcal{B} \arrow{r}{f} \arrow[swap]{dr}{\pi} & \partial W' \setminus f( \mathcal{B} )\arrow{d}{\pi '} \\
     & S^{n-2}
  \end{tikzcd}
\]
Then $\mathcal{W} \cup_f \mathcal{W}' = \cup_{i=1}^n (W_i \cup_f W_i ')$ is an $n$--section of $W \cup_f W'$. 
\end{Theorem}
\begin{proof}
It is immediate that $W_{1...n} \cup_f W_{1...n}'$ is a smooth closed surface. As each $W_I \cup_f W_I'$ is obtained from $(P \times B^{n- \lvert I \rvert -1}) \times [-1,1]$ by gluing $1$--handles along $(P \times B^{n- \lvert I \rvert -1}) \times \lbrace -1 \rbrace$ and $P \times (B^{n- \lvert I \rvert -1}) \times \lbrace 1 \rbrace$, it is an $(n- \lvert I \rvert +2)$--dimensional handlebody with corners. Moreover, the stratification of $\partial (W_I \cup_f W_I')$ by the interiors of the $W_J \cup_f W_{J}'$'s, for $J \supset I$, is inherited from the relative stratification of $W_I$ and $W_I'$ and is precisely the one required by Definition \ref{def_multisection}.   
\end{proof} 
Theorem \ref{thm:gluing} adapts immediately to the PL category.  
\begin{Remark}
Theorem \ref{thm:gluing} also adapts to the case where two relatively multisected manifolds are glued along only some of their boundary components. This results in a relative multisection of the compact manifold obtained by gluing.   
\end{Remark}
For the rest of this section, we keep the notations of Theorem \ref{thm:gluing}.
\begin{Remark}
\label{rmk_bdy_cpnts}
We can compute the genera of the central surface and pieces of $\mathcal{W} \cup_f \mathcal{W}'$, using the parameters of $\mathcal{W}$ and $\mathcal{W}'$. If the central surfaces of $\mathcal{W}$ and $\mathcal{W}'$ have genus $g_1$ and~$g_2$, and $b$ boundary components, we obtain a central surface of genus $g_1 + g_2 + b -1$ for $\mathcal{W} \cup_f \mathcal{W}'$. Further, if the common page of $\mathcal{W}$ and $\mathcal{W}'$ is connected and has genus~$p$, the result of gluing together two pieces of genus $k_1$ and $k_2$ is a handlebody of genus $k_1 + k_2 - (2p + b - 1)$. 
\end{Remark}
Now we describe how to glue relative multisection diagrams. A \emph{cut system of arcs} for a connected compact surface $P$ with non-empty boundary is a collection of disjoint, properly embedded arcs on $P$, such that cutting $P$ along these arcs results in a disk. A cut system of arcs for a disconnected surface with no closed component is a union of cut systems of arcs, one for each component. 
\par Let $\lbrace \Sigma, (\alpha^i)_{1 \leq i \leq n} \rbrace$ and $\lbrace \Sigma', (\beta^i)_{1 \leq i \leq n} \rbrace$ be two diagrams associated to $\mathcal{W}$ and $\mathcal{W}'$. We denote by $C_i$ (resp. $C_{i}'$) the compression body $\cap_{j \neq i} W_i$ (resp. $\cap_{j \neq i} W_{i}'$), and by $\Sigma_i$ (resp.~$\Sigma_{i}'$) the surface resulting of surgery on $\Sigma$ (resp. $\Sigma'$) along the $\alpha^i$ (resp. $\beta^i$) curves. There is an identification between $\Sigma_i$ and $\partial_- C_i$ (resp. between $\Sigma_{i}' $ and $\partial_- C_{i}'$). Therefore $f$ induces a diffeomorphism $f_i$ between $\Sigma_i$ and $\Sigma_{i}'$. This sets the stage for the following proposition.
\begin{Proposition}
\label{prop_glue_diag}
We keep the notations above. For every $i$, we denote by $a^i$ a cut system of arcs for $\Sigma_i$, and by $\beta^{i}_{ns}$ the set of non-separating curves in $\beta^i$. Then $\lbrace \Sigma \cup_f \Sigma', (\alpha^i, \beta^{i}_{ns}, a^i \cup f_i(a^i))_{1 \leq i \leq n} \rbrace$ is a multisection diagram associated to $\mathcal{W} \cup_f \mathcal{W}'$. 
\end{Proposition}
\begin{proof}
Let $1 \leq i \leq n$. We construct a defining cut system for $H_i = C_i \cup_f C_{i}'$, on $ \partial H_i = \Sigma \cup_f \Sigma'$. Denote by $g$ and $g'$ the genus of $\Sigma$ and $\Sigma'$, and by $b$ their common number of boundary components. Let $p$ be the sum of the genera of the $\ell$ connected components of the page~$P$ of the diagrams. By construction, $\Sigma \cup_f \Sigma'$ is a closed surface of genus $g+g'+b-1$. Then~$\alpha^i$ (resp. $\beta^i$) consists in $g-p$ (resp. $g'-p$) non-separating curves and $\ell - 1$ separating curves. Every non-separating curve in $\alpha^i$ and $\beta^i$ bounds a disk in $H_i$. Surgering $\Sigma \cup_f \Sigma'$ along these $g+g'-2p$ curves results in a surface of genus $2p + b - 1$. 
\par Each of the $2p + b- \ell$ closed curves in $a^i \cup f_i(a^i)$ also bounds a disk in $H_i$. To prove this statement, consider an arc $a$ in $a^i$. Identify $H_i$ with the result of gluing $1$--handles to $(\Sigma_i \times [-1,0]) \cup_{f_i} (\Sigma_{i}' \times [0,1])$, along $\Sigma_i \times \lbrace -1 \rbrace$ and $\Sigma_{i}' \times \lbrace 1 \rbrace$. Then, $a \cup f_i(a)$ bounds the properly embedded disk $(a \times [-1,0]) \cup (f(a) \times [0,1])$ in $H_i$. 
\par After further surgering $\Sigma \cup_f \Sigma'$ along the $a^i \cup f_i(a^i)$ curves, we obtain a surface of genus~$\ell - 1$. Now we can use the $\ell -1$ separating curves in $\alpha^i$ (we could use symmetrically the separating curves in $\beta^i$), which also bound disks in $H_i$, to surger this surface into $S^2$. This gives us a cut system on $\Sigma \cup_f \Sigma'$ defining $H_i$.
\end{proof}
We define an \emph{arced diagram} as a tuple $\lbrace \Sigma, (\alpha^i, a^i) \rbrace$, where $\lbrace \Sigma, (\alpha^i) \rbrace$ is a relative multisection diagram and each $a^i$ is a cut system of arcs for $\Sigma_i$. 
\par We now give an algorithm to compute the cut systems of arcs $a^i$ and $f_i (a^i)$, using only the relative multisection diagrams and the action of $f$ on a single set of arcs (for instance,~$a^1$). This is an adaption to higher dimensions of the corresponding algorithm defined for relative trisection diagrams in \cite{castro2018diagrams}.
\par For every $i$ and $j$, the relative fibration on $\partial W$ (resp. $\partial W'$) induces a diffeomorphism~$\phi_{ij}$ (resp.~$\phi_{ij}'$) between $\Sigma_i$ and $\Sigma_{j}$ (resp. $\Sigma_{i}'$ and $\Sigma_{j}'$). Moreover, we have $\phi_{ij}' \circ f_i = f_j \circ \phi_{ij}$. Therefore~$f_j$ is determined by the data of $f_i$, and of the diffeomorphisms $\phi_{ij}$ and $\phi_{ij}'$. These diffeomorphisms can be computed from each relative multisection diagram as follows. 
\par We compute, for instance, $\phi_{ij}$. Let $a^i$ be a cut system of arcs on $\Sigma_i$. We handleslide~$a^i$ over $\alpha^i$ and $\alpha^j$ over $\alpha^j$ until we obtain a new set of arcs $a^j$ (handleslide equivalent to $a^i$) and curves $\tilde{\alpha}^j$ (handleslide equivalent to $\alpha^j$) that do not intersect. Therefore $a^j$ cuts $\Sigma_j$ into a disk. Then $\phi_{ij}$ is the unique (up to isotopy) diffeomorphism that sends $a^i$ to $a^j$. 
\par We iterate this procedure and derive a cut system of arcs $a^2$ for $\Sigma_2$ from a cut system of arcs~$a^1$ for $\Sigma_1$, then a cut system of arcs $a^3$ for $\Sigma_3$ from $a^2$, and so on until obtaining $a^n$ from~$a^{n-1}$. Let $\lbrace b^1,...,b^n \rbrace$ be cut systems of arcs on $\Sigma'$, obtained by applying the same algorithm to $\lbrace \Sigma', (\beta^1,...,\beta^n) \rbrace$, initiating with $b^1 = f_1(a^1)$. Then, for every $i$, $a^{i} \cup b^{i} = a^i \cup f_i(a^i)$. 
\begin{Remark} 
\label{rmk_standard}
We can always produce the advertised cut systems, because each subdiagram $\lbrace \Sigma, (\alpha^{i}, \alpha^{i+1})\rbrace$ and $\lbrace \Sigma', (\beta^{i}, \beta^{i+1})\rbrace$ is handleslide equivalent to the connected sum of $P \simeq \partial_- C_i$ with a number of genus--$1$ Heegaard diagrams of $S^1 \times S^2$ and $S^3$ (see Remark \ref{rmk_standard_rel_diag}).
\end{Remark}  
\begin{Remark}
We can perform the algorithm one more time, and handleslide $a^{n}$ over~$\alpha^{n}$ and $\alpha^{1}$ over $\alpha^{1}$ to get a new set of arcs $\tilde{a}^{1}$ and curves $\tilde{\alpha}^{1}$ that are disjoint, and finally handleslide $\tilde{a}^{1}$ and~$\tilde{\alpha}^{1}$ over $\tilde{\alpha}^{1}$ to obtain a new cut system $(\alpha^{1},c^{1})$. This allows to compare~$c^{1}$ to $a^{1}$ in $\Sigma_{1}$. Keeping the notations of the proof of Theorem \ref{thm_rel_diagram}, the monodromy of the open-book corresponding to the preimage of the loop $[p_1,...,p_n,p_1]$ in $S^{n-2}$ is the diffeomorphism that sends $c^1_{k}$ to $a^1_k$ for all $k$, as \cite[Theorem 5]{castro2018diagrams} states that this map is independent of the choices made while performing the algorithm. If $n=3$, we can compute a non-trivial monodromy, whereas in higher dimension, the monodromy obtained over any loop joining a number of $p_i$'s is necessarily isotopic to the identity.
\end{Remark}
\begin{Example}
\label{ex_gluing}
As an easy application, we recover the genus--$1$ multisection of $S^1 \times S^n$ of Example~\ref{ex_s1sn} by trivially gluing two copies of the relatively multisected $I \times S^n$ of Example~\ref{ex_multi_i_times_sn}. An associated diagram is obtained by gluing two copies of the diagram of Example \ref{ex_diag_s1_sn}. 
\end{Example}
\subsection{Boundary connected sum of relative multisections}
Given two relatively trisected $4$--manifolds $X$ and $X'$, \cite[Theorem 3.20]{castro2022relative} describes an algorithm which produces a relative trisection of $X \natural X'$, where the induced open book on $\partial (X \natural X')$ can be chosen to be any given Murasugi sum of the open books induced on $\partial X$ and $\partial X'$. This algorithm readily adapts to higher dimensions.
\begin{Theorem}
\label{thm_bcs_algo}
Let $n \geq 4$. Given two relatively $n$--sected $(n+1)$--manifolds $(W, \mathcal{W})$ and $(W', \mathcal{W'})$, there exists an explicit algorithm which derives a relative $n$--section for $W \natural W'$ from $\mathcal{W}$ and $\mathcal{W'}$. Moreover, we can choose the induced relative fibration on $\partial (W \natural W') \simeq \partial W \# \partial W'$ to be any given Murasugi sum of the relative fibrations on $\partial W$ and $\partial W'$.
\end{Theorem}
\begin{proof}
We replicate the proof of \cite[Theorem 3.20]{castro2022relative}. We start with diagrams $\lbrace \Sigma , (\alpha^i)_{1 \leq i \leq n} \rbrace$ (with page $P$) and $\lbrace \Sigma' , (\beta^i)_{1 \leq i \leq n} \rbrace$ (with page $P'$), associated respectively to $\mathcal{W}$ and $\mathcal{W}'$. Consider a Murasugi sum of the fibrations on $\partial W$ and $\partial W'$, along arcs $a \subset P$ and $b \subset P'$. We construct a diagram $\lbrace \Sigma \ast \Sigma', (\gamma^i) \rbrace$ as follows. Let $a_1$ be a copy of $a$ on $\Sigma \setminus \alpha^1$. Let $b_1$ be a copy of $b$ on $\Sigma' \setminus \beta^1$. We perform the Murasugi sum~$\Sigma \ast \Sigma'$ along $a_1$ and $b_1$, and we set $\gamma^1 = \alpha^1 \cup \beta^1$. 
\par As the subdiagram $\lbrace \Sigma, \alpha^1, \alpha^2 \rbrace$ is handleslide equivalent to a model diagram described in Remark \ref{rmk_standard_rel_diag}, we can handleslide $a_1$ along $\alpha^1$ curves to obtain an arc $a_2$, and $\alpha^2$ curves along $\alpha^2$ curves to obtain a cut system $\tilde{\alpha}^2$, so that $a_2$ is disjoint from $\tilde{\alpha}^2$. We choose a diffeomorphism of~$\Sigma$, sending $a_1$ to $a_2$, which we extend to $\Sigma \ast \Sigma'$ by the identity on $\Sigma'$ minus a collar of $b_1$. We denote by $f_{12}$ this diffeomorphism. Now the cut systems $\tilde{\alpha}^2$ and $f_{12}(\beta^2)$ are disjoint in $\Sigma \ast \Sigma'$ and we set $\gamma^2 = \tilde{\alpha}^2 \cup f_{12}(\beta^2)$. We can iterate the process until constructing a valid abstract relative \mbox{$n$--section diagram} $\lbrace \Sigma \ast \Sigma', (\gamma^i) \rbrace$ (relying on the fact that each subdiagram $\lbrace \Sigma, \tilde{\alpha}^i, \alpha^{i+1} \rbrace$ is handleslide equi\-valent to a model diagram described in Remark \ref{rmk_standard_rel_diag}). Because we started from diagrams associated to $W$ and $W'$, this diagram defines a multisected $(n+1)$--manifold. An easy generalization of the argument in \cite[Theorem~3.20]{castro2022relative} shows that this manifold is indeed~$W \natural W'$, and that the decomposition induced on the boundary of $W \natural W'$ is the requested Murasugi sum.        
\end{proof}  
\section{Existence in dimension 5}
\label{section_existence}
\begin{Theorem}
\label{existence_relative_trisections}
Let $W$ be a smooth, compact, oriented $(n+1)$--manifold, whose boundary admits a relative fibration. Then, for $2 \leq n \leq 4$, there exists a relative $n$--section inducing this relative fibration.
\end{Theorem}
Note that, because a PL $(n+1)$--manifold admits a smooth structure if $n \leq 6$, this result remains valid in the PL category. We can find a proof of the $3$--dimensional case by \textcite{juhasz2006holomorphic}. The $4$--dimensional case was originally proven by \textcite[Theorem 20]{gay2016trisecting}. A fully detailed proof, explicitly covering the case of a disconnected boundary, was written by \textcite[Theorem 5]{castro2016relative} afterwards. Before addressing the $5$--dimensional case, we start by giving a proof in dimension $4$, based on \cite{castro2016relative} and \cite{lambert2021trisections}. The interest of this proof is that it is rather compact, and provides great insight into the $5$--dimensional case. We need the following lemma, which applies in every dimension, and will be used in both proofs (dimension $4$ and $5$).
\begin{Lemma}
\label{lemma_morse_bdy}
Let $W$ be an $(n+1)$--dimensional compact smooth manifold, whose boundary admits a relative fibration $(\mathcal{B}, \pi)$, with page $P$. Then there exists a closed tubular neighborhood $\nu \mathcal{B}$ of $\mathcal{B}$ in~$\partial W$, and a smooth function $f:W \rightarrow [0,n+1]$, verifying the following conditions.
\begin{itemize}
\item Denoting by $S^{n-2} = B_0 \cup (\partial B_0 \times I) \cup B_1$ a decomposition of $S^{n-2}$ into two \mbox{$(n-2)$--balls} and a thickening of the boundary of the first ball, $f^{-1}(\lbrace 0 \rbrace) \cap \partial W = \pi^{-1}(B_0) \setminus Int(\nu \mathcal{B}) \simeq P \times B^{n-2}$ and  $f^{-1}(\lbrace n+1 \rbrace) \cap \partial W = \pi^{-1}(B_1) \setminus Int(\nu \mathcal{B}) \simeq P \times B^{n-2}$;
\item for $0 < t < n+1$, $f^{-1}(\lbrace t \rbrace) \cap \partial W \simeq (P \times \partial B^{n-2}) \cup (\partial P \times B^{n-2})$;
\item $f$ is a self-indexing Morse function in the interior of $W$, with no critical points of index $0$ and $n+1$, and no critical points in a collar of $\partial W$.
\end{itemize}    
\end{Lemma}
\begin{proof}
We parametrize $B^{n-1}$ by $(r,\theta) \in [0,1] \times S^{n-2}$. By Theorem \ref{prop_rel_trivial}, there exists a closed tubular neighborhood $\nu \mathcal{B}$ of $\mathcal{B}$ in~$\partial W$, parametrized by $\phi:\mathcal{B} \times B^{n-1} \rightarrow \nu \mathcal{B}$, such that the restriction of $\pi$ to $\nu \mathcal{B} \setminus \mathcal{B}$ is the projection on the angular coordinate of~$B^{n-1}$. We denote by $\tilde{\pi}$ the restriction of $\pi$ to $\partial W \setminus Int(\nu \mathcal{B})$. We extend~$\tilde{\pi}$ to a smooth map from~$\partial W$ to $B^{n-1}$, by setting $\pi \circ \phi(x,r,\theta)=(r,\theta)$ for $(x,r,\theta) \in \mathcal{B} \times B^{n-1}$. We choose an arc in~$\partial B^{n-1}$ and project~$B^{n-1}$ onto this arc, identified with $[0,n+1]$ (see Figure~\ref{fig_sproj_d2} for a $2$--dimensional example). This amounts to identify $B^{n-1}$ with the hypercube $I_1 \times ... \times I_{n-1}$ and project on one edge. We call $F$ the composition of this projection with the extension of $\tilde{\pi}$. This function, from $\partial W$ to $[0,n+1]$, verifies:
\begin{itemize}
\item $F^{-1}(\lbrace 0 \rbrace) \simeq F^{-1}(\lbrace n+1 \rbrace) \simeq P \times B^{n-2}$;
\item $F^{-1}(\lbrace t \rbrace) \simeq (P \times \partial B^{n-2}) \cup (\partial P \times B^{n-2})$, for $0 < t < n+1$;
\item $F$ has no critical points as it is locally a projection. 
\end{itemize}
\begin{figure}[h!]
\[
\begin{tikzpicture}[scale=0.2]
\node (mypic) at (0,0) {\includegraphics[scale=0.3]{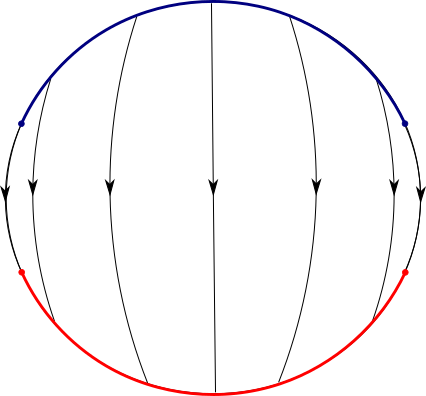}};
\end{tikzpicture}
\]
\caption{Projection of $D^2$ on the red arc}
\label{fig_sproj_d2}
\end{figure} 
We extend $F$ to a collar neighborhood of $\partial W$ by $F(x,t) = F(x)$ for $(x,t) \in \partial W \times I$, without creating critical points. We further extend $F$ to $W$, and take a Morse approximation of the resulting function. As there are no critical points in $\partial W \times I$, this approximation agrees with $F$ on a collar of $\partial W$. For the same reason, we can modify this Morse function away from a collar of $\partial W$ to obtain a self-indexing function, that agrees with $F$ on this collar. We call $f$ this approximation, which we assume has no critical points of index $0$ and $n+1$: this is the requested function. 
\end{proof}
\begin{Theorem}[Gay and Kirby]
\label{existence_dim4} 
Let $W$ be a smooth, compact, oriented $4$--manifold with non-empty boundary, together with a relative fibration of $\partial W$. Then $W$ admits a relative trisection that induces this relative fibration.
\end{Theorem}
\begin{proof}
Let $f$ be a smooth function from $W$ to $[0,4]$, as defined in Lemma \ref{lemma_morse_bdy}. We choose a gradient-like vector field $\chi$ on $W$, so that the couple $(f, \chi)$ is Morse--Smale in the interior of $W$.
\par The level set $f^{-1}(\lbrace 3/2 \rbrace)$ is diffeomorphic to a connected sum, whose summands are the components of $P \times [a,b]$ (with $[a,b]$ some interval) and a number of copies of $S^1 \times S^2$ (corresponding to the non-connecting $1$--handles of $f$). It is a compact $3$--manifold, together with a natural sutured decomposition of its boundary: $\partial (f^{-1}(\lbrace 3/2 \rbrace)) \simeq (P \times \lbrace a \rbrace) \cup (\partial P \times [a,b]) \cup (P \times \lbrace b \rbrace)$. The interior of $f^{-1}(\lbrace 3/2 \rbrace)$ intersects the descending manifolds of the critical points of index $2$ of $f$ along a link $L$, whose components are the attaching spheres of the $2$--handles. We denote by $\nu_L$ a closed tubular neighborhood of this link in $f^{-1}(\lbrace 3/2 \rbrace)$. 
\par We construct a particular sutured Heegaard splitting of $f^{-1}(\lbrace 3/2 \rbrace)$, inducing the sutured decomposition of $\partial f^{-1}(\lbrace 3/2 \rbrace)$ described above. We want one of the compression bodies of this splitting to be the connected union of the components of $\nu_L$, a collar of $P \times \lbrace a \rbrace$, and $1$--handles. As $f^{-1}(\lbrace 3/2 \rbrace)$ is connected, we can connect the components of $\nu_L$ and the components of a collar neighborhood of $P \times \lbrace a \rbrace$ (denoted by $P \times [a,a']$) with $3$--dimensional $1$--handles. Let $A$ be the resulting tridimensional compression body, with negative boundary $P \times \lbrace a \rbrace$. We consider the manifold $\overline{f^{-1}(\lbrace 3/2 \rbrace) \setminus A}$, with the sutured decomposition of its boun\-dary between $P \times \lbrace b \rbrace$ and $\partial \overline{A \setminus (P \times [a,a'])}$. We choose (see, for instance, \cite{dissler2023relative}) a sutured Heegaard splitting of $\overline{f^{-1}(\lbrace 3/2 \rbrace) \setminus A}$ corresponding to this decomposition of the boundary. One of the compression bo\-dies of this sutured Heegaard splitting is a collar neighborhood of $\partial \overline{A \setminus (P \times [a,a'])}$ with \mbox{$1$--handles} attached. We denote by $C_1$ the union of $A$ and this compression body, and by $C_2$ the other compression body of the sutured Heegaard splitting (see Figure \ref{fig_dec_level_set}). By construction, $C_1$ and $C_2$ are $3$--dimensional compression bodies, with negative boundary $P \times \lbrace a \rbrace$ or $P \times \lbrace b \rbrace$, and positive boundary the Heegaard surface of the sutured Heegaard splitting. Moreover, as wanted, $C_1$ is the connected union of the components of $\nu_L$, a collar of $P \times \lbrace a \rbrace$, and $1$--handles.
\par Recall that $\nu_L$ is a tubular neighborhood of the attaching link of the $4$--dimensional $2$--handles. Therefore, by flowing up $\chi$, we can thicken $C_2$ to a $4$--dimensional subset of $f^{-1}([3/2,5/2])$, without hitting the $4$--dimensional $2$--handles. We identify this subset with the cylinder $C_2 \times [3/2,5/2]$. We set $W_1 = f^{-1}([0,3/2]) \cup (C_2 \times [3/2,2])$, and $W_3 = f^{-1}([5/2,4]) \cup (C_2 \times [2,5/2])$. As $W_1$ is isotopic to $f^{-1}([0,3/2])$ and $W_3$ is isotopic to $f^{-1}([5/2,4])$, $W_1$ and $W_3$ are $4$--dimensional compression bodies, with negative boundary diffeomorphic to the product of $P$ with an interval. We denote by~$W_2$ the closure of their complement in $W$ (see Figure \ref{fig_morse_function_dim4} for a schematic of the decomposition). We claim that~$W_2$ is also a $4$--dimensional handlebody: it is diffeomorphic to the union of $C_1 \times I$ with $2$--handles, but all these $2$--handles are canceled by $1$--handles in $C_1 \times I$. 
\par As for the tridimensional intersections, $W_{13} \simeq C_2$, $W_{12} \simeq C_1$. The remaining tridimensional intersection, $W_{23}$, is the result of surgery on $C_1$ along $L$. Recall that $C_1$ is a boundary connected sum of a tubular neighborhood of $L$ with a compression body (with negative boundary a copy of~$P$). Therefore the result of the surgery is again a boundary connected sum, with summands a tubular neighborhood of a link of circles and a tridimensional collar of $P$: it is a compression body with negative boundary a copy of $P$. Finally, $W_{123}$ is isotopic to the Heegaard surface of the sutured Heegaard splitting of $f^{-1}(\lbrace 3/2 \rbrace)$. Therefore $W=\cup_{i=1}^3 W_i$ is a relative trisection of $W$, inducing the required open-book decomposition on $\partial W$.  
\end{proof}
\begin{figure}[h!]
\[
\begin{tikzpicture}[scale=0.3]
\node (mypic) at (0,0) {\includegraphics[scale=0.3]{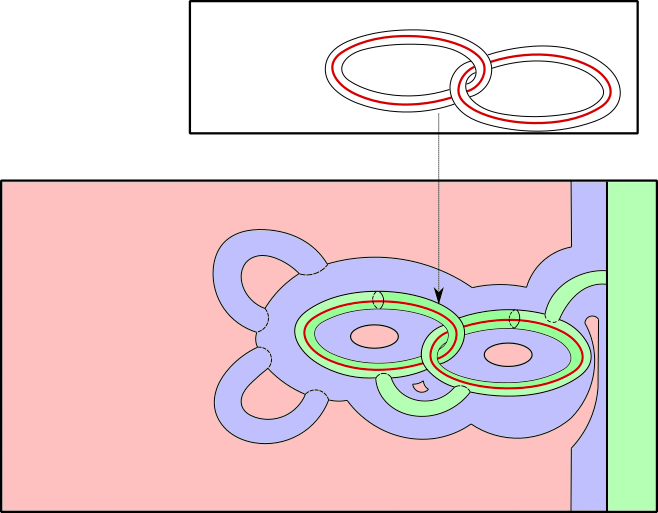}};
\node[text=red] (mypic) at (-5,-3) {$C_2$};
\node[text=blue] (mypic) at (5,-4.25) {$C_1$};
\node (mypic) at (2,1.3) {$\chi$};
\node (mypic) at (-6.2,5) {$f^{-1}(\lbrace 2 \rbrace)$};
\node[text=red] (mypic) at (0,-0.8) {$L$};
\node[text=green] (mypic) at (8,-3.9) {$A$};
\node (mypic) at (-11.8,-2) {$f^{-1}(\lbrace 3/2 \rbrace)$};
\node (mypic) at (8.7,-7.7) {$a$};
\node (mypic) at (7.4,-7.5) {$a'$};
\node (mypic) at (-8.7,-7.5) {$b$};
\end{tikzpicture}
\]
\caption{Decomposition of the level set $f^{-1}(\lbrace 3/2 \rbrace)$}
\label{fig_dec_level_set}
\end{figure}
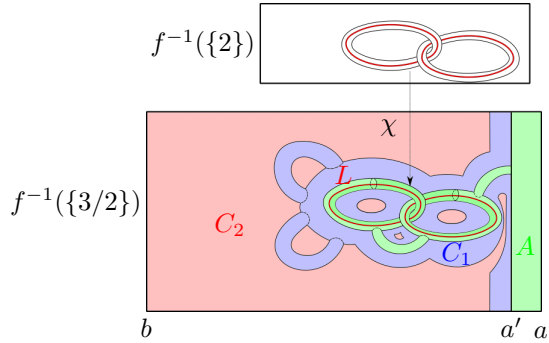 
\begin{figure}[h!]
\[
\begin{tikzpicture}[scale=0.3]
\node (mypic) at (0,0) {\includegraphics[scale=0.3]{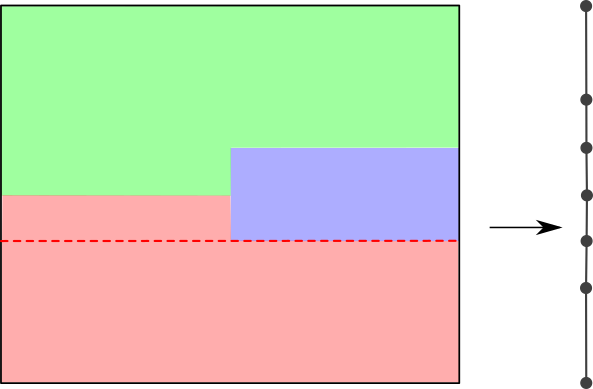}};
\node (mypic) at (-3,-3) {$W_1$};
\node (mypic) at (2,0) {$W_2$};
\node (mypic) at (-3,3) {$W_3$};
\node (mypic) at (5.9,0) {$f$};
\node (mypic) at (8.9,5) {$4$};
\node (mypic) at (8.9,2.5) {$3$};
\node (mypic) at (8.9,1.3) {$2.5$};
\node (mypic) at (8.9,0) {$2$};
\node (mypic) at (8.9,-1.3) {$1.5$};
\node (mypic) at (8.9,-2.4) {$1$};
\node (mypic) at (8.9,-5) {$0$};
\node[text=red] (mypic) at (-11.2,-1.3) {$f^{-1}(\lbrace 1.5 \rbrace)$};
\end{tikzpicture}
\]
\caption{Schematic for a relative trisection obtained from a Morse function}
\label{fig_morse_function_dim4}
\end{figure}
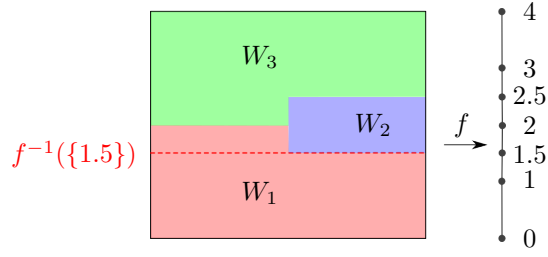 
We now prove the existence of smooth relative multisections in dimension $5$. We will adapt and combine arguments in \cite{castro2016relative} (existence in dimension $4$) and \cite{aribi2023multisections} (existence in dimension 5 for closed smooth manifolds). We will proceed as in the $4$--dimensional case, by finding an adequate decomposition of an intermediate level set, and using this decomposition to build our pieces. Only this time, we need a so-called \emph{special relative quadrisection} of the $4$--dimensional level set. This is the object of the following definition and lemma.  
\begin{Definition}
\label{def_special_relative}
Let $X$ be a compact $4$--manifold with non-empty boundary, together with two closed (possibly disconnected) surfaces $L_2$ and $L_3$ embedded in $Int(X)$, transverse to each other. A \emph{special relative quadrisection} of $X$ with respect to $L_2$ and $L_3$ is a decomposition $X=\cup_{i=1}^4 X_i$, where:
\begin{itemize}
\item $X_i$ is a $4$--dimensional compression body, with $\partial_{-} X_i = X_i \cap \partial X \simeq P \times I$, where $P$ is a given compact surface with no closed component;
\item $X_{i,i+1}$ is a $3$--dimensional compression body, with $\partial_{-} X_{i,i+1} = X_{i,i+1} \cap \partial X \simeq P$;
\item $X_{1234} = \Sigma$, a connected compact surface, with $\partial \Sigma = X_{1234} \cap \partial X \neq \emptyset$;
\item $X_{ij} = \Sigma$ if $j \notin \lbrace i-1, i+1 \rbrace$;
\item $X_1 \cup X_2$ (resp. $X_2 \cup X_3$) is a tubular neighborhood of $L_2$ (resp. $L_3$) with $1$--handles attached.
\end{itemize}
\end{Definition}
See Figure \ref{fig_spec_rel_quad} for a schematic.
\begin{figure}[h!]
\[
\begin{tikzpicture}[scale=0.3]
\node (mypic) at (0,0) {\includegraphics[scale=0.5]{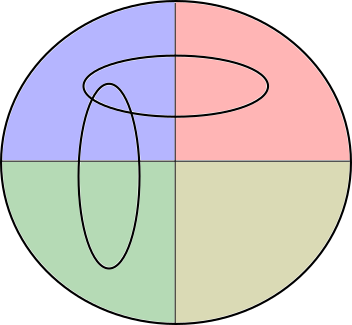}};
\node (mypic) at (-3,-1) {$L_2$};
\node (mypic) at (-1,3.5) {$L_3$};
\node (mypic) at (-8,4) {$X_2$};
\node (mypic) at (8,4) {$X_3$};
\node (mypic) at (-8,-4) {$X_1$};
\node (mypic) at (8,-4) {$X_4$};
\end{tikzpicture}
\]
\caption{Schematic for a special relative quadrisection with respect to $L_2$ and $L_3$}
\label{fig_spec_rel_quad}
\end{figure}
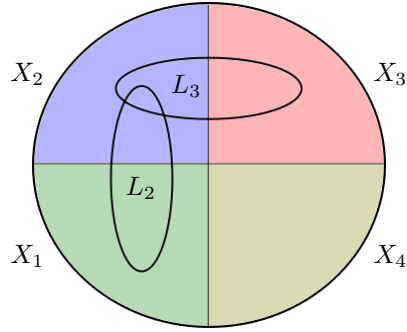 
\begin{Lemma}
\label{lemma_quadrisection}
Let $X$ be a connected, compact, orientable smooth $4$--manifold with non-empty boun\-dary, together with two closed surfaces $L_2$ and $L_3$ embedded in $Int(X)$, transverse to each other. Then, given an open-book decomposition of $\partial X$, there exists a special relative quadrisection of $X$ with respect to $L_2$ and $L_3$ inducing this decomposition.    
\end{Lemma}
\begin{proof}
First, we proceed as in the proof of Lemma \ref{lemma_morse_bdy}. We extend the projection of the open-book decomposition of $\partial X$ to a smooth function $F$, from a collar neighborhood of $\partial X$ to $[0,4]$, with no critical points. Then we consider a tubular neighborhood $\nu_{L_2 \cup L_3}$ of $L_2 \cup L_3$ in $X$, and gradually extend $F$ to a self-indexing Morse function $f$ from $X \setminus Int(\nu_{L_2 \cup L_3})$ to $[0,4]$, such that $f(\partial \nu_{L_2 \cup L_3} ) = 0.5$. We assume that $f$ has no critical points of index $0$ and~$4$. We choose a gradient-like vector field $\chi$ on $X$ so that the pair $(f,\chi)$ is Morse--Smale in the interior of $X$. 
\par We pick a collection of $2$--disks, properly embedded in $\nu_{L_2 \cup L_3}$, such that each disk intersects transversally $L_2 \cup L_3$ in a single point, and such that removing a tubular neighborhood of each disk from $\nu_{L_2 \cup L_3}$ results in a collection of (possibly disjoint) \mbox{$4$--dimensional} handlebodies~$\nu'$. Such a collection exists because a tubular neighborhood of a surface with non-empty boundary is a tubular neighborhood of \mbox{$2$--dimensional} $0$-- and $1$--handles, therefore a \mbox{$4$--dimensional} handlebody, as we work with orientable manifolds; and these handlebodies intersect along $4$--balls corresponding to the transverse points of intersection of the surfaces. Note that we can get $\nu_{L_2 \cup L_3}$ back from $\nu'$ by attaching $2$--handles along $L_2'=L_2 \cap \nu'$ and $L_3' = L_3 \cap \nu'$. Because $X$ is connected, so must be $f^{-1}([0,1.5])$, as only $1$--handles are connecting. We can assume that the attaching regions of the $1$--handles are disjoint from $\overline{\nu \setminus \nu'}$. Then the union of $f^{-1}([0,1.5])$ and $\nu'$ is a connected union of handlebodies and $1$--handles, therefore a $4$--dimensional handlebody, that we denote by $\nu''$. By construction, the intersection of this handlebody with $\partial X$ is diffeomorphic to $P \times I$.
\par Now we focus on the descending spheres of the critical points of index $2$. By general position, we consider that they lie in $\partial \nu'' \setminus (\partial L_2' \cup \partial L_3')$. We decompose each of the corresponding $2$--handles as the union of a $1$--handle and two $2$--handles, each one canceling the $1$--handle. Finally we define~$X_2$ to be $\nu''$ with these $1$--handles attached.
\par We give $X_2$ a structure of compression body, by setting $\partial_{-}X_2 = \partial X_2 \cap \partial X$. In $\partial_+ X_2$ lies the link $\partial L_2' \cup \partial L_3'$ and the attaching spheres of the pairs of $2$--handles defined in the previous step, that we denote by $M_2$ and $M_3$. Note that $\partial_+ X_2$ is a compact $3$--manifold diffeomorphic to the connected sum of a collar neighborhood of $f^{-1}(\lbrace 1.5 \rbrace) \cap \partial X$ in $f^{-1}(\lbrace 1.5 \rbrace)$, and copies of $S^1 \times S^2$ corresponding to the $4$--dimensional $1$--handles and handlebodies attached while constructing $X_2$. This manifold admits a sutured decomposition of its boundary: $\partial (\partial_+ X_2) = f^{-1}(\lbrace 1.5 \rbrace) \cap \partial X \simeq (P \times \lbrace a \rbrace \cup \partial P \times [a,b] \cup P \times \lbrace b \rbrace)$, with $[a,b]$ some interval.      
\par We choose a sutured Heegaard splitting: $\partial_+ X_2 = X_{12} \cup X_{23}$, such that $X_{12}$ (resp. $X_{23}$) is a tubular neighborhood of $\partial L_2' \cup M_2 \cup (P \times \lbrace a \rbrace)$ with $1$--handles attached (resp. a tubular neighborhood of $\partial L_3' \cup M_3 \cup (P \times \lbrace b \rbrace)$ with $1$--handles attached). Such a sutured Heegaard splitting can be obtained as follows (see Figure \ref{fig_spec_quad_level_set} for a schematic). Take a tubular neighborhood of $\partial L_2' \cup M_2$ in $\partial_+ X_2$. Connect the components of this neighborhood with $1$--handles. Connect the resulting handlebody and a collar neighborhood $P \times [a,a']$ with another $1$--handle, thus obtaining a handlebody $H_1$. Proceed the same way with $L_3' \cup M_3$ and a collar neighborhood $P \times [b',b]$, denote by $H_2$ the resulting handlebody. Then the closure of $\partial_+ X_2$ minus these two handlebodies is a compact $3$--manifold, with the sutured decomposition of its boundary induced by the sutured decomposition of $\partial (\partial_+ X_2)$. Therefore (see, e.g., \cite{dissler2023relative}), there exists a sutured Heegaard splitting inducing this decomposition. One of the compression bodies of this sutured Heegaard splitting is a collar neighbohood of $\partial \overline{H_1 \setminus P \times [a,a']}$ with $1$--handles attached. Define $X_{12}$ as the union of $H_1$ and this compression body. Define $X_{23}$ symmetrically, as the union of $H_2$ and the other compression body. By construction, $X_{12}$ and $X_{23}$ are compression bodies, with negative boundary $P \times \lbrace a \rbrace$ or~$P \times \lbrace b \rbrace$. The two pieces intersect along their positive boundary, which we denote by $\Sigma$. 
\par Note that the attaching spheres of the $2$--handles defined in the previous steps can all be seen as embedded as the cores of $1$--handles in $X_{12}$ or $X_{23}$. We now define $X_1$ (see Figure~\ref{fig_spec_quad} for a schematic). We push $X_{12}$ upward along $\chi$, above the level set~$f^{-1}(\lbrace 2 \rbrace)$. We obtain a new subset of~$X$, which is a thickening of $X_{12}$ with $2$--handles attached along $M_2$. By construction, each of these $2$--handles is canceled against some $1$--handle in the thickening of $X_{12}$. We define $X_1$ as the union of this new subset with the \mbox{$2$--handles} corresponding to the tubular neighborhoods of the disks removed from $L_2$ when constructing~$X_2$. Again, by construction, these $2$--handles, attached along $L_2'$, are canceled against $1$--handles in the thickening of $X_{12}$. Therefore $X_1$ is a compression body, with negative boundary $\partial_- X_1 = X_1 \cap \partial X \simeq P \times I$. The positive boundary of $X_1$ is the union of $X_{12}$ and $X_{14} = \overline{\partial_+ X_1 \setminus X_{12}}$. We claim that $X_{14}$ is a compression body. We consider~$X_{12}$ as the boundary connected sum of a tubular neighborhood of a link with a compression body (with negative boundary a copy of $P$). In this perspective, $X_{14}$ is obtained from $X_{12}$ by surgery along this link. It is again a boundary connected sum of a collar neighborhood of a copy of $P$ with a tubular neighborhood of a link, with $1$--handles attached. Therefore~$X_{14}$ is a compression body, with same positive boundary as $X_{12}$, and negative boundary diffeomorphic to $P$. Observe that, as required by the definition, $X_{1} \cup X_{2}$ is a tubular neighborhood of $L_{2}$ with $1$--handles attached.
\par We define $X_3$ and $X_{24}$ similarly, by thickening $X_{23}$ and attaching $2$--handles along $\partial L_3' \cup M_3$. When thickening $X_{23}$, we can ensure that $X_1$ and~$X_3$ intersect only along $X_{12} \cap X_{23} = \Sigma$. The remaining piece $X_4 = \overline{X \setminus (X_1 \cup X_2 \cup X_3)}$ is also a compression body, as it is built from $(P \times I) \times I$ by adding $1$--handles (corresponding to the critical points of index $3$, looking upside down), with the required decomposition of its boundary.   
\end{proof}
\begin{figure}[h!]
\[
\begin{tikzpicture}[scale=0.3]
\node (mypic) at (0,0) {\includegraphics[scale=0.5]{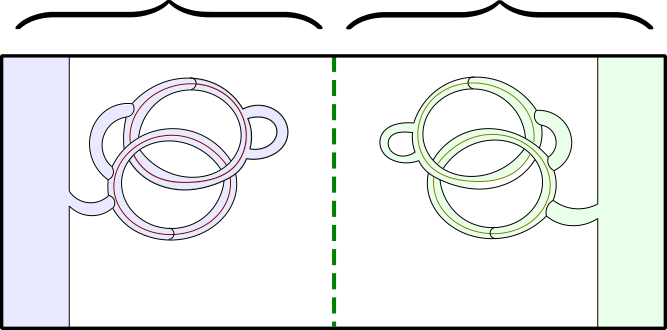}};
\node[text=blue] (mypic) at (-13,-1) {$H_2$};
\node[text=green!40!gray] (mypic) at (13,-1) {$H_1$};
\node[draw=black, fill=white] (mypic) at (0,-3.5) {Heegaard surface};
\node (mypic) at (8.3,4) {$\partial L_2'$};
\node (mypic) at (8,-3.9) {$M_2$};
\node (mypic) at (-9.1,4) {$\partial L_3'$};
\node (mypic) at (-8.1,-3.9) {$M_3$};
\node (mypic) at (7.6,8) {$X_{12}$};
\node (mypic) at (-7.7,8) {$X_{23}$};
\node (mypic) at (14.7,-8) {$a$};
\node (mypic) at (11.7,-7.8) {$a'$};
\node (mypic) at (-11.7,-7.8) {$b'$};
\node (mypic) at (-14.7,-7.8) {$b$};
\end{tikzpicture}
\]
\caption{Schematic for the decomposition of $\partial_+ X_2$}
\label{fig_spec_quad_level_set}
\end{figure}
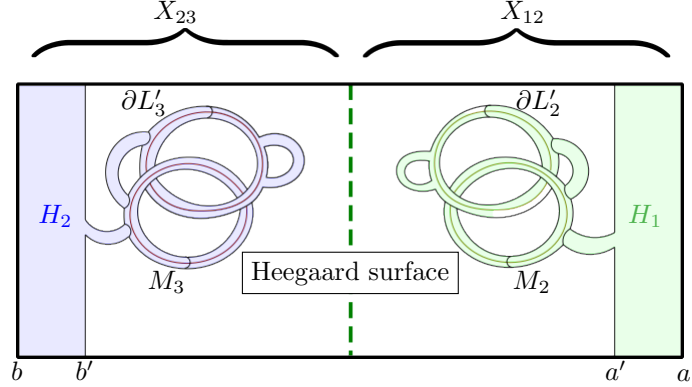 
\begin{figure}[h!]
\[
\begin{tikzpicture}[scale=0.3]
\node (mypic) at (0,0) {\includegraphics[scale=0.3]{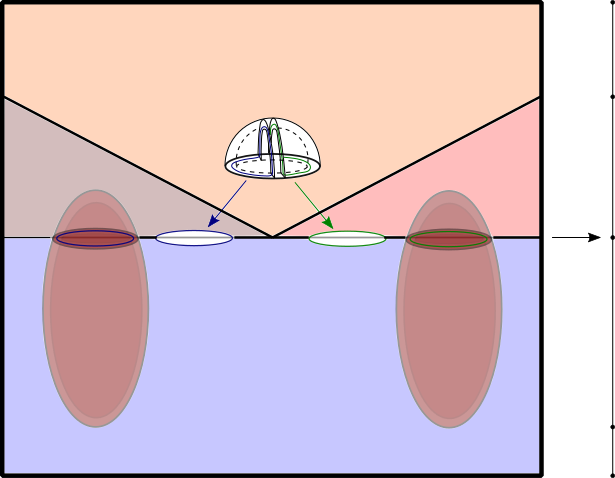}};
\node (mypic) at (5.4,2.5) {$X_1$};
\node (mypic) at (-7,2.5) {$X_3$};
\node (mypic) at (-1,5.5) {$X_4$};
\node (mypic) at (-1,-5.5) {$X_2$};
\node (mypic) at (7,1.1) {$f$};
\node (mypic) at (9.2,6.3) {$4$};
\node (mypic) at (9.2,4) {$3$};
\node (mypic) at (9.2,0) {$1.5$};
\node (mypic) at (9.3,-4.9) {$0.5$};
\node (mypic) at (8.9,-6.2) {$0$};
\node[text=red, font=\footnotesize] (mypic) at (-10.7,0) {$f^{-1}(\lbrace 1.5 \rbrace)$};
\node (mypic) at (-5.5,-4) {$L_3$};
\node (mypic) at (4,-4) {$L_2$};
\node[text=blue, font=\footnotesize] (mypic) at (-5.5,-1) {$\partial L_3 '$};
\node[text=green!60!black, font=\footnotesize] (mypic) at (4,-1) {$\partial L_2 '$};
\node[text=blue, font=\footnotesize] (mypic) at (-2.5,-1) {$M_3$};
\node[text=green!60!black, font=\footnotesize] (mypic) at (1.5,-1) {$M_2$};
\node[font=\footnotesize, fill=white] (mypic) at (-1,4) {$2$--handle};
\end{tikzpicture}
\]
\caption{Schematic for a special relative quadrisection, obtained from a Morse function}
\label{fig_spec_quad}
\end{figure}
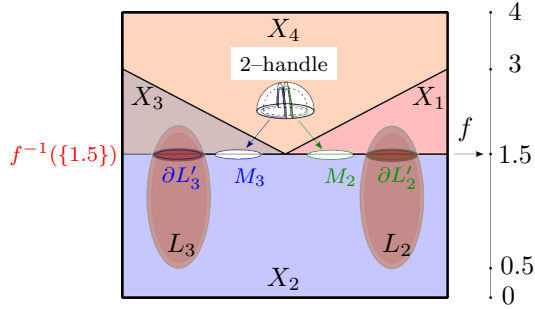 
\begin{Theorem}
\label{thm_relative_multisection}
Let $W$ be a compact, smooth, connected, oriented $5$--manifold with non-empty boundary, such that $\partial W$ rela\-tively fibers over $S^2$. Then there exists a relative $4$--section of $W$ inducing this relative fibration on $\partial W$.
\end{Theorem}
\begin{proof}
Let $f$ be a self-indexing Morse function from $W$ to $[0,5]$, as defined in Lemma \ref{lemma_morse_bdy}. We choose a gradient-like vector field $\chi$ on $W$, so that the couple $(f, \chi)$ is Morse--Smale in the interior of $W$.
\par Consider the level set $X=f^{-1}(\lbrace 2.5 \rbrace)$. It is a compact, connected $4$--manifold, whose boun\-dary $\partial X = X \cap \partial W \simeq (P \times S^1) \cup (\partial P \times D^2)$ admits an open-book decomposition induced by the relative fibration on $\partial W$. In the interior of $X$, we have the two links of $2$--spheres, denoted by $L_2^+$ and $L_3^-$, corresponding to the ascending spheres of all critical points of index $2$ and the des\-cending spheres of all critical points of index $3$. As $(f,\chi)$ is Morse--Smale, $L_2^+$ and $L_3^-$ intersect transversally. Therefore, according to Lemma~\ref{lemma_quadrisection}, the level set $X$ admits a special relative quadrisection with respect to $L_2^+$ and $L_3^-$. We denote by $X=W_{24} \cup W_{23} \cup W_{13} \cup W_{14}$ this decomposition. Remember that this ensures that $W_{24} \cup W_{23}$ (resp. $W_{23} \cup W_{13}$) is a tubular neighborhood of $L_{2}^+$ (resp. $L_{3}^-$) in $X$, with $1$--handles attached. 
\par Now we build the relative quadrisection of $W$ as follows (see Figure \ref{fig_morse_5} for a schematic). We push down $W_{13} \cup W_{14}$ along $\chi$ until reaching $f^{-1}(\lbrace 1.5 \rbrace)$. We define $W_1$ as the union of this subset of $f^{-1}([1.5,2.5])$ with $f^{-1}([0,1.5])$. As the $2$--handles of $f$ are attached along $L_{2}^+ \subset W_{24} \cup W_{23}$,~$W_1$ is isotopic to $f^{-1}([0,1.5])$. Therefore it is the connected union of a collar neighborhood of $P \times D^2$ with $1$--handles, i.e. a compression body with negative boundary $W_1 \cap \partial W \simeq P \times D^2$. Then we set $W_2 = \overline{f^{-1}([0,2.5]) \setminus W_1}$. By flowing down $\chi$, we have that $W_2$ is diffeomorphic to $(W_{24} \cup W_{23}) \times I$ with \mbox{$3$--handles} attached along~$L_{2}^+$ (these are exactly the $2$--handles of $f$ seen upside down). As each of these \mbox{$3$--handles} caps off exactly one of the attaching spheres in~$L_{2}^+$,~$W_2$ is a $5$--dimensional compression body, with negative boundary $W_2 \cap \partial W \simeq \big((W_{24} \cup W_{23}) \cap \partial W \big) \times I \simeq P \times D^2$. The positive boundary of $W_2$ is divided between $W_{24} \cup W_{23}$ and a new part $W_{12} = \overline{\partial W_2 \setminus (\partial_- W_2 \cup W_{23} \cup W_{24})}$. This new part is obtained from $W_{24} \cup W_{23}$ by surgery along~$L_{2}^+$. This surgery replaces each tubular neighborhood of a copy of $S^2$ by a tubular neighborhood of a circle. Therefore $W_{12}$ is a $4$--dimensional compression body, with negative boundary $\partial_{-} W_{12}= W_{12}  \cap \partial W \simeq P \times I$. Now the situation is symmetric: we define $W_3$ and~$W_4$ in $f^{-1}([2.5,5])$, attaching $3$--handles along $L_3^-$. The $3$-- and $2$--dimensional intersections res\-pect the definition of a $5$--dimensional relative quadrisection, because they are exactly the $3$-- and \mbox{$2$--dimensional} pieces of the special relative quadrisection of $f^{-1}({\lbrace 2.5 \rbrace})$. 
\end{proof}
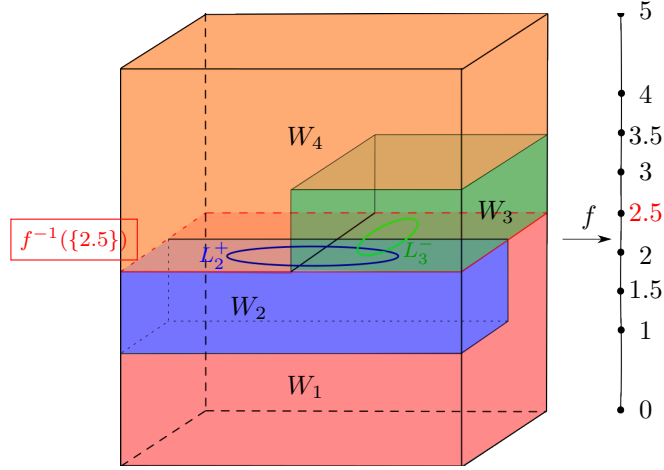
\begin{figure}[h!]
\[
\begin{tikzpicture}[scale=0.3]
\node (mypic) at (0,0) {\includegraphics[scale=0.3]{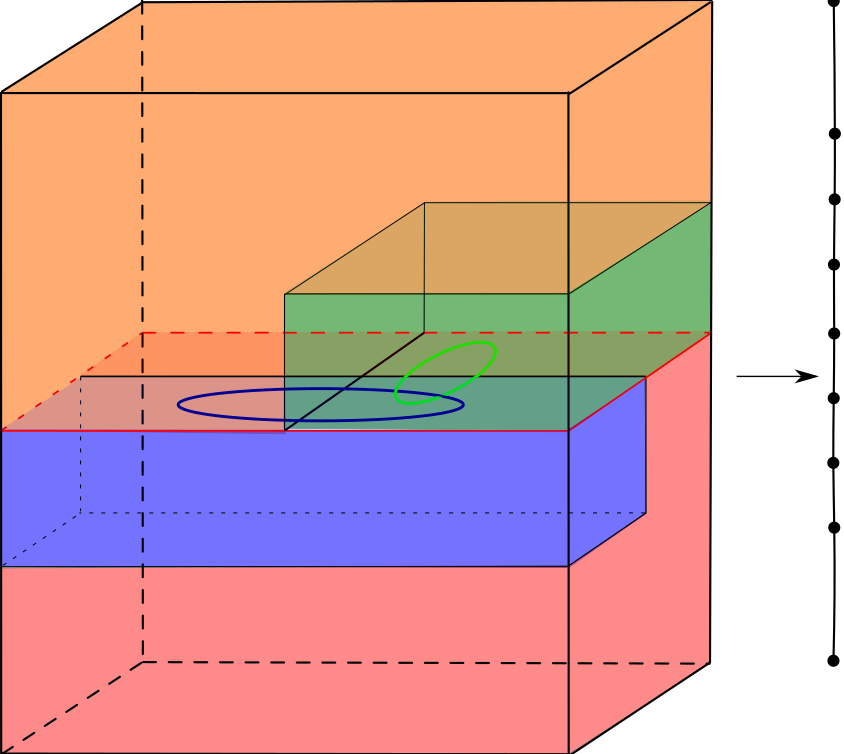}};
\node (mypic) at (5.4,1.2) {$W_3$};
\node (mypic) at (-5.5,-3) {$W_2$};
\node (mypic) at (-3,4.5) {$W_4$};
\node (mypic) at (-3,-6.5) {$W_1$};
\node (mypic) at (9.5,1) {$f$};
\node (mypic) at (12,10.1) {$5$};
\node (mypic) at (12,6.4) {$4$};
\node (mypic) at (12,4.6) {$3.5$};
\node (mypic) at (12,3) {$3$};
\node[text=red] (mypic) at (12,1.2) {$2.5$};
\node (mypic) at (12,-0.7) {$2$};
\node (mypic) at (12,-2.2) {$1.5$};
\node (mypic) at (12,-4) {$1$};
\node (mypic) at (12,-7.5) {$0$};
\node[text=red, font=\footnotesize, draw=red] (mypic) at (-13.2,0) {$f^{-1}(\lbrace 2.5 \rbrace)$};
\node[text=blue, font=\footnotesize] (mypic) at (-7,-0.7) {$L_{2}^+$};
\node[text=green!60!black, font=\footnotesize] (mypic) at (2,-0.5) {$L_{3}^-$};
\end{tikzpicture}
\]
\caption{Schematic for a relative quadrisection obtained from a Morse function}
\label{fig_morse_5}
\end{figure} 
\section{More Examples}
\label{section_more_examples}
\subsection{More examples of relative multisections}
\subsubsection{Relative multisections of disk bundles over spheres}
\label{ex_multi_sn1_times_d2} 
We adapt the construction of relative trisections of $2$--disk bundles over $S^2$, from \cite{castro2018diagrams}, to any $2$--disk bundle $E$ over $S^{n-1}$, with $n \geq 4$. Note that $E$ is necessarily a trivial bundle. Indeed, using a clutching construction, we see that $E$ is determined by an element of $\pi_{n-2}(\mathrm{Diff}_0(D^2))$, where $\mathrm{Diff}_0(D^2)$ is the identity component of the diffeomorphism group of $D^2$. By \textcite{earle1970teichmuller}, this group is homotopy equivalent to $SO(2)$: therefore $\pi_{n-2}(\mathrm{Diff}_0(D^2))$ is trivial for $n \geq 4$. In particular, there is no obstruction for $E$ to admit a relative multisection, as $\partial E = S^1 \times S^{n-1}$.
\par We refer to Figure \ref{fig_snd2} for a schematic of the construction of the multisection (with $n=4$). We denote by $D \simeq D^2$ the fiber of the bundle. We denote by $s$ the stereographic projection from $S^{n-1} \setminus \lbrace (0,...,0,1) \rbrace$ to $\mathbb{R}^{n-1}$, and by $\mathbb{R}^{n-1}= \cup_{i=1}^n A_i$ a standard $n$-decomposition of~$\mathbb{R}^{n-1}$. Let $B_i= s^{-1}(A_i) \cup \lbrace (0,...,0,1) \rbrace$. Then the decomposition $S^{n-1} = \cup_{i=1}^n B_i$ is such that $\cap_{i \in I} B_{i}$ is an $(n- \lvert I \rvert)$--ball for $1 \leq \lvert I \rvert \leq n-1$, and $\cap_{i=1}^n B_i \simeq S^0$. 
\par We consider a collection of $n$ disjoint disks $\lbrace D_1,..., D_n \rbrace$, in the interior of $D$. For $1 \leq i \leq n$, we set $W_i = (B_i \times \overline{D \setminus D_i}) \cup_{B_{i,i+1} \times D_{i+1}}  (B_{i+1} \times D_{i+1})$. Therefore $W_i$ is an $(n+1)$--dimensional compression body of genus $1$, with negative boundary $\partial_{-} W_i = \partial D \times B_i \simeq (S^1 \times I) \times B^{n-2}$.  
\par The computation of the $W_I$'s, for $\lvert I \rvert \leq n$, is straightforward. We find, for $I= \lbrace 1,...,n \rbrace$, a surface constituted of two copies of $D$ with the interior of the $D_i$'s removed, connected by~$n$ cylinders glued along the $\partial D_i$'s (that is, a compact surface of genus $n-1$ with $2$ boundary components). For $\lvert I \rvert =n-1$, we find a $3$--dimensional compression body, with negative boundary a cylinder, and positive boundary $W_{1...n}$. For $2 \leq \lvert I \rvert \leq n-2$, we find an $(n-\lvert I \rvert +2)$--dimensional compression body of genus $\lvert I \rvert +1$, with negative boundary the product of an $(n- \lvert I \rvert -1)$--ball with a cylinder, and its positive boundary multisected according to Remark \ref{rmk_dec_pos_bdy}. Therefore we have produced a relative $n$--section of $E$. 
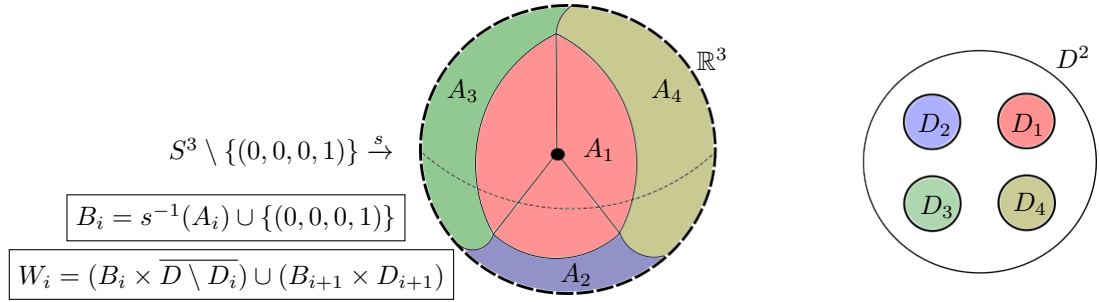
\begin{figure}[h!]
\[
\begin{tikzpicture}[scale=0.3]
\node (mypic) at (0,0) {\includegraphics[scale=0.5]{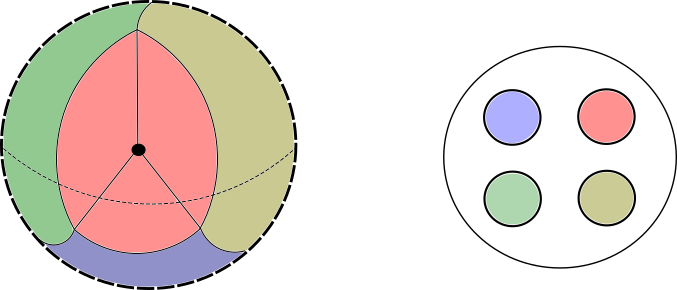}};
\node (mypic) at (-7,0) {$A_1$};
\node (mypic) at (-13,2.6) {$A_3$};
\node (mypic) at (-4,2.6) {$A_4$};
\node (mypic) at (-8,-5.5) {$A_2$};
\node (mypic) at (-2,4) {$\mathbb{R}^3$};
\node (mypic) at (-21,0) {$S^3 \setminus \lbrace (0,0,0,1) \rbrace \xrightarrow{s}$};
\node (mypic) at (13.9,4.1) {$D^2$};
\node (mypic) at (11.9,1.1) {$D_1$};
\node (mypic) at (7.8,1.1) {$D_2$};
\node (mypic) at (12,-2.4) {$D_4$};
\node (mypic) at (7.9,-2.5) {$D_3$};
\node[draw=black] (mypic) at (-23,-3) {$B_i = s^{-1}(A_i) \cup \lbrace (0,0,0,1) \rbrace$};
\node[draw=black] (mypic) at (-23.2,-5.5) {$W_i = (B_i \times \overline{D \setminus D_i}) \cup (B_{i+1} \times D_{i+1})$};
\end{tikzpicture}
\]
\caption{Schematic for a quadrisection of $S^3 \times D^2$}
\label{fig_snd2}
\end{figure} 
By adapting \cite[Example 5.1]{castro2018diagrams} to higher dimensions, we can derive a relative multisection diagram for $E$ from this relative multisection (see Figure \ref{fig_diag_snd2} for examples when $n=4$). 
\begin{figure}[h!]
\[
\begin{tikzpicture}[scale=0.3]
\node (mypic) at (0,0) {\includegraphics[scale=0.47]{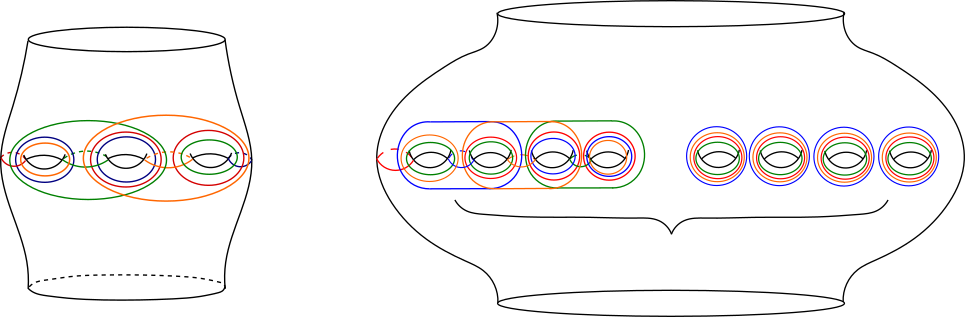}};
\node (a) at (8,-3.6) {\small genus $n-1$};
\node (a) at (7.5,0) {...};
\end{tikzpicture}
\]
\caption{A relative multisection diagram for $S^{3} \times D^2$ (left), for $S^{n-1} \times D^2$ (right, only four sets of curves represented)}
\label{fig_diag_snd2}
\end{figure}
\subsubsection{Relative multisections of compact fiber bundles over the circle}
\label{ex_s1X}
In dimension $4$, trisections of closed fiber bundles over the circle were originally constructed by \textcite{gay2016trisecting}, then by \textcite{koenig2017trisections}. Relative trisections of compact fiber bundles over the circle were constructed by the author in \cite{dissler2023relative}. Multisections of suitable closed fiber bundles over the circle were constructed in any dimension by \textcite{moussard2025multisections}. Our goal is to adapt the methods developed in these references to construct relative multisections of suitable compact fiber bundles over the circle.
\par For $n \geq 4$, let $W$ be an orientable fiber bundle over the circle, with fiber a compact \mbox{$n$--manifold}~$X$. Such a bundle is diffeomorphic to $\big( X \times [0,1] \big) / \big( (x,0) \sim (\phi(x),1) \big)$, where the \emph{monodromy} $\phi$ is a diffeomorphism of~$X$. The idea is to use a relative multisection of the fiber, whose pieces are preserved by the monodromy. The existence of this decomposition is not guaranteed (especially as we cannot even assume that a fiber of dimension greater than $5$ admits a relative multisection). In what follows, we assume that $X$ admits such a decomposition. 
\begin{Lemma}
\label{lemma_bundle_bdy}
Let $n \geq 4$. If there exists a fibration $\#^m S^1 \times S^{n-2} \rightarrow \#^{\ell} S^1 \times S^{n-1} \rightarrow S^1$, then~$m=0$ and~$\ell=1$.
\end{Lemma}
\begin{proof}
Using the long exact sequence of the fibration, we have that $\pi_{n-2} (\#^m S^1 \times S^{n-2}) \simeq \pi_{n-2} (\#^{\ell} S^1 \times S^{n-1})$. But (see, for instance, \cite[Chapter $6$]{kosinski2013differential}) $\pi_{n-2}(\#^{\ell} S^1 \times S^{n-1}) \simeq \pi_{n-2}(W_{\ell})$, with $W_{\ell}$ a genus--$\ell$ handlebody of dimension $n+1$. As $\pi_{n-2}(W_{\ell})$ is trivial, so is $\pi_{n-2} (\#^m S^1 \times S^{n-2})$, which implies that $m=0$. In this case, the exact sequence gives $\pi_{1}(\#^{\ell} S^1 \times S^{n-1})  \simeq \mathbb{Z}$, therefore~$\ell =1$.  
\end{proof}
Lemma \ref{lemma_bundle_bdy} implies that, if both $X$ and $W$ admit a relative multisection, then the components of $\partial W$ are copies of $S^1 \times S^{n-1}$ and the components of $\partial X$ are spherical.
\begin{Proposition}
\label{prop_bundles1}
Let $n \geq 4$. Let $W$ be a smooth fiber bundle over $S^1$, with monodromy~$\phi$ and fiber a compact $n$--manifold~$X$ with non-empty boundary. Suppose that~$X$ admits a relative $(n-1)$--section $X=\cup_{i=1}^{n-1} X_i$, such that, for every $i$, $\phi(X_i)=X_i$. Then~$W$ admits a relative $n$--section.     
\end{Proposition}   
\begin{proof}
Denote by $b$ the number of components of $\partial X$, which are all spheres, by Lemma \ref{lemma_bundle_bdy}. If~$n \geq 5$, then the page of a relative fibration of~$\partial X$ is necessarily a disjoint union of $b$ disks (see Remark~\ref{coro_fibration}). If $n=4$, we choose this page to be a disjoint union of~$b$ disks: we decompose $\partial X$ using $b$ copies of the open-book decomposition of $S^3$ with disk pages, then extend this decomposition to a relative trisection of $X$ using Theorem \ref{existence_dim4}. We denote by~$\Sigma$ the central surface of the relative multisection of~$X$, with genus~$g$ and necessarily $b$ boundary components. We decompose~$W$ according to the table on Figure \ref{fig_tables1X}. The horizontal part is the segment $[0,1]$, identified at $0 \simeq 1$ to produce~$S^1$. The vertical part represents~$X$, divided between the $X_i$'s. The top row corresponds to $X_1 \times S^1$, the second row to $X_2 \times S^1$, and so on until the bottom row, which corresponds to $X_{n-1} \times S^1$. Each row is divided between subsets of the $W_j$'s.
\par To fix ideas, consider that $n=4$. Let $0<t_1<...<t_k<...<t_{6}<1$ be such that each $t_k$ is the intersection of $[0,1]$ with one of the vertical segments in the table. Choose~$6b$ disjoint points in~$\Sigma \cap \partial X$ ($6$ in each of the $b$ components). By considering a small enough neighborhood of each of these points, we obtain a collection of $6$ sets of $b$ disjoint balls, denoted by $\lbrace B_k \rbrace_{1\leq k \leq 6}$, such that each ball intersects transversally each $X_I$ into an $(n- \lvert I \rvert +1)$--ball, and intersects $X_I \cap \partial X$ into an $(n- \lvert I \rvert)$--ball. Then we use these balls to connect the different pieces of each $W_j$. For instance, $W_1$ is equal to the union $ ([t_{1},t_{6}] \times X_{1}) \cup \big( ([t_{6},1] \times B_1)\cup_{\phi}([0,t_{1}] \times B_1) \big)$, minus the sets $Int({B}_{k}) \times [t_{k-1},t_k]$ for $k \neq 1$, and $W_2$ is equal to $ ([t_{1},t_{2}] \times B_{2}) \cup ([t_{2},t_{5}] \times X_{2}) \cup ([t_{5},t_{6}] \times B_{6}) \cup \big( ([t_{6},1] \times X_1)\cup_{\phi}([0,t_{1}] \times X_1) \big)$, minus the sets $Int({B}_{k}) \times [t_{k-1},t_k]$ for $k \neq 2,6$, and $(Int({B}_{1}) \times [t_{6},1]) \cup_{\phi} (Int({B}_{1}) \times [0,t_1])$. Note that, for every $I$ and~$k$, $X_I \setminus Int({B}_k)$ is isotopic to $X_I$. 
\par We claim that $W= \cup_{i=1}^4 W_i$ is a relative multisection of $W$. Indeed, $W_{1234}$ is a properly embedded surface, equivalent to $6$ copies of $\Sigma$ (one at each $t_k$), with bands joining $\Sigma \times \lbrace t_{k-1} \rbrace$ to $\Sigma \times \lbrace t_{k} \rbrace$. A tridimensional intersection $W_{ijk}$ consists in copies of $\Sigma \times I$ and copies of~$X_{12}$, $X_{13}$ or $X_{23}$, connected by tridimensional tubes. It is a tridimensional compression body, with positive boundary $W_{1234}$, and negative boundary $W_{ijk} \cap \partial W$ diffeomorphic to disjoint sets of disks connected by bands, i.e. to a disconnected union of annuli, which we denote by $P$. A \mbox{$4$--dimensional} intersection $W_{ij}$ consists in copies of $X_{12} \times I$, $X_{13} \times I$ or $X_{23} \times I$, and copies of $X_1$, $X_2$, or $X_3$, all connected by tubes. It is a $4$--dimensional compression body with negative boundary diffeomorphic (modulo corners) to $P \times I$. Finally each $W_i$ consists in copies of $X_1 \times I$, $X_2 \times I$ or $X_3 \times I$, connected by tubes. It is a $5$--dimensional compression body, with negative boundary diffeomorphic (modulo corners) to $P \times B^2$. We have thus produced a relative quadrisection of $W$.
\par The $(n+1)$--dimensional case is just a straightforward generalization. We obtain a relative $n$--section of $W$, with central surface diffeomorphic to $2(n-1)$ copies of $\Sigma = X_{1...n-1}$ connected by bands, and page diffeomorphic to a disjoint set of annuli. Note that the page of any relative multisection of $W$ is necessarily a disconnected union of annuli, because each component of $\partial W$ is diffeomorphic to $S^1 \times S^n$, which admits a unique relative fibration (see Remark \ref{coro_fibration}).
\end{proof}
\begin{figure}[h!]
\[
\begin{tikzpicture}[scale=0.3]
\node (mypic) at (0,0) {\includegraphics[scale=0.4]{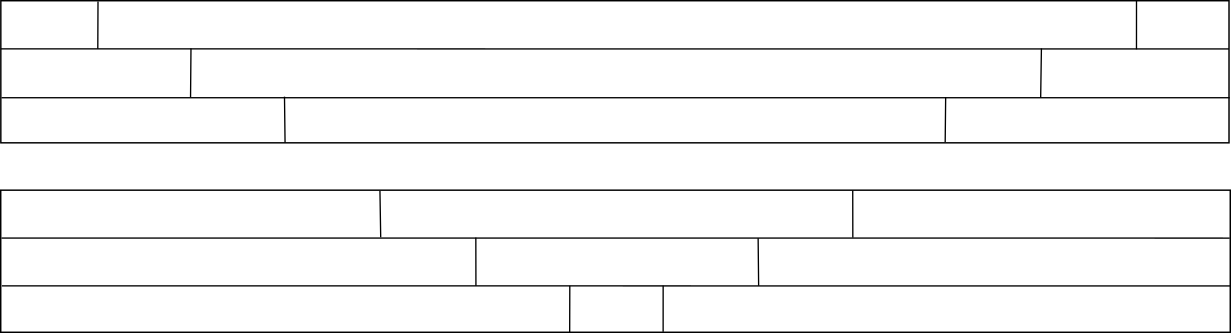}};
\node (mypic) at (-20,5) {$W_2$};
\node (mypic) at (0,5) {$W_1$};
\node (mypic) at (20,5) {$W_2$};
\node (mypic) at (23.5,5) {\small $X_1$};

\node (mypic) at (-18,3.2) {$W_3$};
\node (mypic) at (0,3.2) {$W_2$};
\node (mypic) at (23.5,3.2) {\small $X_2$};
\node (mypic) at (18,3.2) {$W_3$};

\node (mypic) at (-16,1.7) {$W_4$};
\node (mypic) at (16,1.7) {$W_4$};
\node (mypic) at (23.5,1.7) {\small $X_3$};
\node (mypic) at (0,1.7) {$W_3$};

\node (mypic) at (15,-1.7) {$W_{n-2}$};
\node (mypic) at (-15,-1.7) {$W_{n-2}$};
\node (mypic) at (23.7,-1.7) {\small $X_{n-3}$};
\node (mypic) at (0,-1.7) {$W_{n-3}$};

\node (mypic) at (-13,-3.3) {$W_{n-1}$};
\node (mypic) at (23.7,-3.3) {\small $X_{n-2}$};
\node (mypic) at (13,-3.3) {$W_{n-1}$};
\node (mypic) at (0,-3.3) {$W_{n-2}$};

\node (mypic) at (0,-5) {$W_{n-1}$};
\node (mypic) at (-12,-5) {$W_{n}$};
\node (mypic) at (23.7,-5) {\small $X_{n-1}$};
\node (mypic) at (12,-5) {$W_{n}$};

\node (mypic) at (-22,-7) {$0$};
\node (mypic) at (22,-7) {$1$};
\node (mypic) at (0,0.3) {\small $\vdots$};
\end{tikzpicture}
\]
\caption{Schematic for a relative multisection of a suitable $(n+1)$--dimensional fiber bundle over $S^1$}
\label{fig_tables1X}
\end{figure}
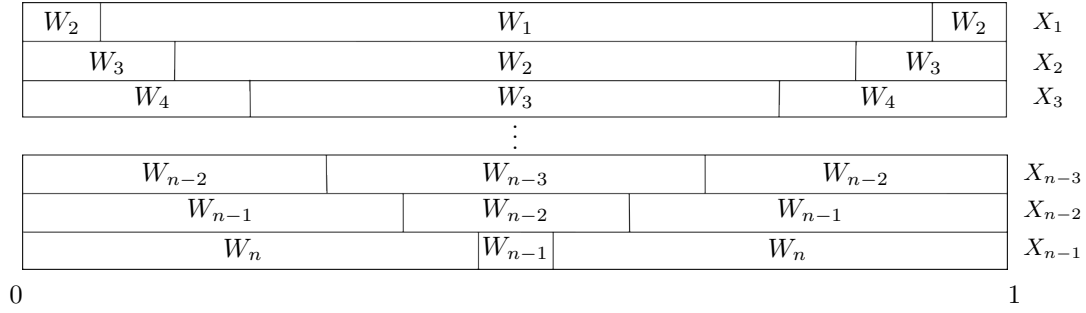 
An associated relative multisection diagram for $W$ can be obtained from the corresponding relative multisection diagram for $X$. The process described below is a direct adaptation of the $4$--dimensional case, described in \cite[Section 5]{dissler2023relative}.
\par When constructing the tridimensional intersections from the central surface, a specific compression body appears: $C_{\Sigma} \simeq \Sigma \times I$. Its positive boundary is the union of $\Sigma \times \lbrace 0 \rbrace$ and~$\Sigma \times \lbrace 1 \rbrace$, with matching boundary components connected two by two by bands. Its negative boundary is a disjoint union of $b$ disks, each one capping off a boundary component of the positive boundary. The construction of a cut system for this compression body is detailed in \cite[Subsection 5.1]{dissler2023relative}, we recall it briefly. Pick a boundary component $\partial_1 \Sigma$ of~$\partial \Sigma$. Choose $2g$ non isotopic and properly embedded arcs on $\Sigma \times \lbrace 0 \rbrace$, with extremities on $\partial_1 \Sigma \times \lbrace 0 \rbrace$, that cut the surface into a disk with~$b-1$ holes. By gluing each arc $a$ with its symmetric on $\Sigma \times \lbrace 1 \rbrace$, we obtain $2g$ non isotopic curves on~$\partial_+ C_{\Sigma}$ (top curves on Figure~\ref{fig_cut_system_csigma}). Choose $(b-1)$ arcs connecting $\partial_1 \Sigma$ to each of the remaining boundary components, then draw the $(b-1)$ curves obtained as the doubles of these arcs. Finally, draw the last $(b-1)$ curves, each one parallel to a boundary component but $\partial_1 \Sigma$. We define a cut system for the quotient compression body $C _{\Sigma,\phi}=(\Sigma \times [t_{2(n-1)},0]) \cup_{\phi} (\Sigma \times [0,t_1])$ (with~$\partial_+ C_{\Sigma, \phi}$ the union of $\Sigma \times \lbrace t_{2(n-1)} \rbrace$ and $\Sigma \times \lbrace t_1 \rbrace$, with matching boundary components connected two by two by bands) in the same fashion: only this time, each arc on $\Sigma \times \lbrace t_{2(n-1)} \rbrace$ is connected with its image by $\phi$ on $\Sigma \times \lbrace t_{1} \rbrace$. 
\par Denote by $\lbrace \Sigma, (\alpha^i)_{1 \leq i \leq n-1} \rbrace$ the relative multisection diagram for $X$. To construct our relative multisection diagram $\lbrace \Sigma_W, (\beta^i)_{1 \leq i \leq n} \rbrace$ associated to the relative multisection of~$W$, we choose an embedding of $\Sigma_W$ as $2(n-1)$ copies of $\Sigma$ with alternating orientations, denoted by $(\Sigma^i)_{1 \leq i \leq 2(n-1)}$, such that each boundary component of $\Sigma^i$ is connected with the matching boundary component of $\Sigma^{i+1}$ by a band. We consider that the monodromy occurs between $\Sigma^{2(n+1)}$ and $\Sigma^1$: therefore each boundary component of $\Sigma^{2(n+1)}$ must be connected with its image by the monodromy on~$\Sigma^1$. A tridimensional piece of the relative multisection of $W$ is a union of copies of tridimensional intersections of the relative multisection of~$X$, and one copy of~$C_{\Sigma}$ (in the case of $W_{1...n-1}$), or one copy of~$C_{\Sigma,\phi}$ (in the case of $W_{2...n}$), or two copies of $C_{\Sigma}$ (all other cases), connected by tubes. Its negative boundary is a union of annuli corresponding to the negative boundaries of these compression bodies, connected by bands. We obtain the $\beta^1$ curves as follows.
\begin{enumerate}
\item For $j \neq 1,2(n-1)$, we draw on $\Sigma^j$ a copy of $\alpha^1$, with the convention that if $j$ is even, then the copy must be the image of $\alpha^1$ by reflection.  
\item We draw on the union of $\Sigma^{2(n-1)}$, $\Sigma^{1}$ and their connecting bands the cut system for~$C_{\Sigma, \phi}$. 
\end{enumerate}
The other sets of curves are obtained by adapting these two steps to each tridimensional piece. 
\begin{figure}[h!]
\[
\begin{tikzpicture}[scale=0.3]
\node (mypic) at (0,0) {\includegraphics[scale=0.4]{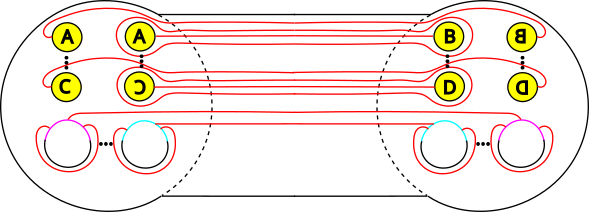}};
\end{tikzpicture}
\]
\caption{A cut system for $C_{\Sigma}$ \\ \footnotesize{The pink and blue arcs on the boundary components are identified by color to form the band connecting these components}}
\label{fig_cut_system_csigma}
\end{figure}
An example of such a diagram is featured on Figure \ref{fig_gen_diag_s1xo_arced}, when the fiber~$X^o$ is a closed $5$--manifold $X$ minus the interior of a ball. Remember that (Example \ref{ex_diag_xn_minus_ball}), if $X$ admits a multisection with associated diagram $\lbrace \Sigma , (\alpha^i)_{1 \leq i \leq n-1} \rbrace $, then~$X^o$ admits a relative multisection with associated diagram $\lbrace \Sigma \setminus Int(D), (\alpha^i)_{1 \leq i \leq n-1} \rbrace $, where $D$ is a disk disjoint from the $(\alpha^i)_{1 \leq i \leq n-1}$ curves. 
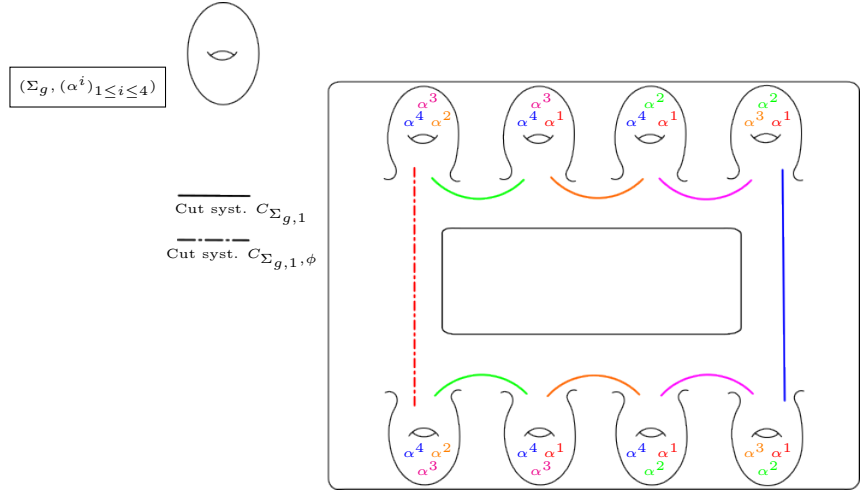
\begin{figure}[h!]
\[
\begin{tikzpicture}[scale=0.6]
\node (mypic) at (0,0) {\includegraphics[scale=0.6]{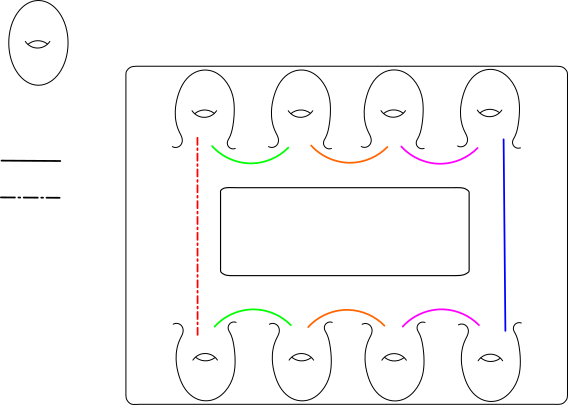}};
\node[minimum size] (mypic) at (-6.1,0.7) {\tiny Cut syst. $C_{\Sigma_{g,1}}$};
\node[minimum size] (mypic) at (-6.1,-0.3) {\tiny Cut syst. $C_{\Sigma_{g,1},\phi}$};
\node[draw, fill=white] (mypic) at (-9.5,3.5) {\tiny $(\Sigma_g, (\alpha^i)_{1 \leq i \leq 4})$};

\node[text=red] (mypic) at (5.8,2.8) {\tiny $\alpha^1$};
\node[text=red] (mypic) at (3.3,2.8) {\tiny $\alpha^1$};
\node[text=red] (mypic) at (0.8,2.8) {\tiny $\alpha^1$};
\node[text=red] (mypic) at (5.8,-4.5) {\tiny $\alpha^1$};
\node[text=red] (mypic) at (3.3,-4.5) {\tiny $\alpha^1$};
\node[text=red] (mypic) at (0.8,-4.5) {\tiny $\alpha^1$};

\node[text=blue] (mypic) at (-2.3,2.8) {\tiny $\alpha^4$};
\node[text=blue] (mypic) at (2.6,2.8) {\tiny $\alpha^4$};
\node[text=blue] (mypic) at (0.2,2.8) {\tiny $\alpha^4$};
\node[text=blue] (mypic) at (-2.3,-4.5) {\tiny $\alpha^4$};
\node[text=blue] (mypic) at (2.6,-4.5) {\tiny $\alpha^4$};
\node[text=blue] (mypic) at (0.2,-4.5) {\tiny $\alpha^4$};

\node[text=green] (mypic) at (5.5,3.2) {\tiny $\alpha^2$};
\node[text=green] (mypic) at (3,3.2) {\tiny $\alpha^2$};
\node[text=green] (mypic) at (5.5,-4.9) {\tiny $\alpha^2$};
\node[text=green] (mypic) at (3,-4.9) {\tiny $\alpha^2$};

\node[text=orange] (mypic) at (-1.7,2.8) {\tiny $\alpha^2$};
\node[text=orange] (mypic) at (-1.7,-4.5) {\tiny $\alpha^2$};
\node[text=orange] (mypic) at (5.2,2.8) {\tiny $\alpha^3$};
\node[text=orange] (mypic) at (5.2,-4.5) {\tiny $\alpha^3$};

\node[text=magenta] (mypic) at (-2,3.2) {\tiny $\alpha^3$};
\node[text=magenta] (mypic) at (0.5,3.2) {\tiny $\alpha^3$};
\node[text=magenta] (mypic) at (-2,-4.9) {\tiny $\alpha^3$};
\node[text=magenta] (mypic) at (0.5,-4.9) {\tiny $\alpha^3$};

\end{tikzpicture}
\]
\caption{Relative $5$--section diagram for $S^1 \times X^o$, obtained from the genus--$g$ $4$--section diagram for $X$ in the top-left corner}
\label{fig_gen_diag_s1xo_arced}
\end{figure}
\begin{Remark}
\cite{moussard2025multisections} also considers the case where the monodromy permutes the pieces of the multisection of the fiber. This can be adapted to relative multisections of compact fiber bundles.
\end{Remark}
\subsection{More applications of the Gluing Theorem}
\label{sub_glue_ex}
\subsubsection{Multisections of $S^{n-1} \times S^2$}
Trivially gluing two copies of relatively $n$--sected $S^{n-1} \times D^2$ (see Example \ref{ex_multi_sn1_times_d2}) produces an \mbox{$n$--section} of $S^{n-1} \times S^2$. An example of an associated multisection diagram for $S^3 \times S^2$ is featured on Figure~\ref{fig_s3s2}. The genus of the multisection is $2n+1$, but the diagram can easily be destabilized $n+1$ times to recover the genus--$n$ multisection diagram by Moussard \cite{moussard2025multisections}. However, this diagrammatic operation is only well-defined for $n \leq 5$ in the smooth category, as shows the following discussion.
\par As in dimension $3$ and $4$, an $n$--section diagram $\mathcal{D}= \lbrace \Sigma, (\alpha^i)_{1 \leq i \leq n} \rbrace $ for a PL--manifold~$W$ can be destabilized for any $n$: if $\mathcal{D}$ contains two groups of $k>0$ and $n-k$ parallel curves, with one curve in each $\alpha^i$, such that a curve of the first group intersects each curve of the second group in exactly one point, then the associated multisection is the result of a stabilization (\cite[Proposition~2.2]{moussard2025multisections}). This is due to the fact that such a diagram can be modified by isotopy to a connected sum of two multisection diagrams $\mathcal{D}_1 \# \mathcal{D}_2$, with~$\mathcal{D}_1$ a genus--$1$ diagram of Example~\ref{ex_multi_diagram_sn}. As a multisection diagram determines a PL multisected manifold up to PL-homeomorphism (Theorem \ref{thm_diagram}), $\mathcal{D}_1$ uniquely determines~$S^{n+1}$. Thus,~$\mathcal{D}_2$ is a multisection diagram for $W$, obtained by destabilizing~$\mathcal{D}$. Practically, if the group of~$k$ curves is disjoint from all the curves of the multisection diagram but the curves of the second group, we destabilize $\mathcal{D}$ by compressing $\Sigma$ along one of the curves of the first group, and deleting all the curves of the second group. In the smooth category, however, this is only true for $n \leq 5$: otherwise a diagram of Example \ref{ex_multi_diagram_sn} cannot be uniquely associated to~$S^{n+1}$.
\begin{figure}[h!]
\[
\begin{tikzpicture}[scale=0.2]
\node (mypic) at (0,0) {\includegraphics[scale=0.3]{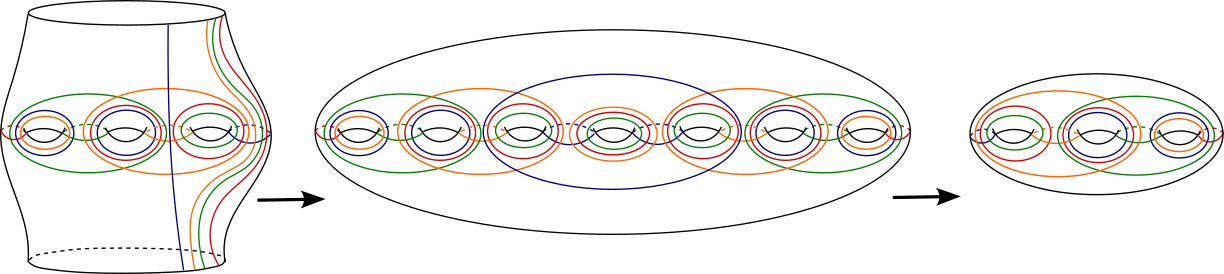}};
\end{tikzpicture}
\]
\caption{Multisection diagram for $S^3 \times S^2$ obtained from the arced diagram for $S^3 \times D^2$ on the left, then destabilized $4$ times}
\label{fig_s3s2}
\end{figure}
\subsubsection{Multisections of $1$--spun manifolds} 
\label{ex_1spun}
In \cite{dissler2024multisections}, the author defines the $m$--spun of a closed $n$--manifold $X$: $\mathcal{S}_m(X) = (X^o \times S^m) \cup_{S^{n-1} \times S^m} (S^{n-1} \times B^{m+1})$, where $X^o$ stands for $X$ minus the interior of a ball. Trisections of $1$--spun \mbox{$3$--manifolds} were originally constructed by \textcite{meier2017trisections}, then \cite{dissler2024multisections} provided multisections for \mbox{$m$--spun} $3$--manifolds, with $m \geq 1$. We focus here on $1$--spun $n$--manifolds, with~$n \geq 3$.
\begin{Proposition}
If $X$ admits a genus--$g$ $(n-1)$--section, then $\mathcal{S}_1(X)$ admits a multisection of genus $(n-1)(2g+1)+1$. Moreover, this multisection can be destabilized once, for $n \leq 5$ in the smooth category, and for any $n$ in the PL-category.
\end{Proposition}
\begin{proof}
Let $\Sigma_g$ be the central surface of the multisection of $X$. Example \ref{ex_xn_minus_ball} provides a relative multisection of $X^o$, with central surface $\Sigma_{g,1}$ and page a disk. From this, Example~\ref{ex_s1X} derives a relative multisection of $X^o \times S^1$, with central surface $\Sigma_{2(n-1)g,2}$ and page an annulus. Moreover, Example \ref{ex_multi_sn1_times_d2} provides a relative multisection of~$S^{n-1} \times D^2$, with central surface $\Sigma_{n-1,2}$ and page an annulus. Trivially gluing these relative multisections results in an $n$--section of $\mathcal{S}_1 (X)$, of genus $(n-1)(2g+1)+1$ (see Remark \ref{rmk_bdy_cpnts}). We obtain the associated multisection diagram by trivially gluing the associated arced diagrams. The resulting diagram for $\mathcal{S}_1(X)$ can be destabilized immediately once (with the restrictions mentioned in the previous example), which produces a multisection of the $1$--spun of $X$ of genus $(n-1)(2g+1)$, for $n \leq 5$ in the smooth category, and for any $n$ in the PL-category. Arced diagrams for $S^1 \times \mathbb({\mathbb{CP}}^2)^o$ and $S^3 \times D^2$ are featured on Figure \ref{fig_prediagram_spun_cp2}. Using these diagrams, we obtain the multisection diagram for $\mathcal{S}_1(\mathbb{CP}^2)$, featured on Figure \ref{fig_spun_cp2}.
\end{proof}  
\begin{figure}[h!]
\[
\begin{tikzpicture}[scale=0.2]
\node (mypic) at (0,0) {\includegraphics[scale=0.5]{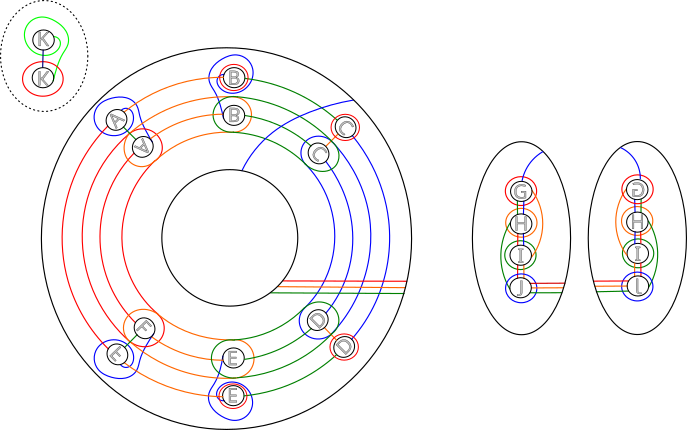}};
\end{tikzpicture}
\]
\caption{Arced diagram for $S^{1} \times (\mathbb{CP}^2)^o$ (left), for $S^{3} \times D^2$ (right)\\ \footnotesize{Top left: trisection diagram for $\mathbb{CP}^2$, from which the diagram for $S^{1} \times (\mathbb{CP}^2)^o$ is derived; the dashed disk represents $S^2$; circles with matching labels are identified}}
\label{fig_prediagram_spun_cp2}
\end{figure}
\begin{figure}[h!]
\[
\begin{tikzpicture}[scale=0.5]
\node (mypic) at (0,0) {\includegraphics[scale=0.6]{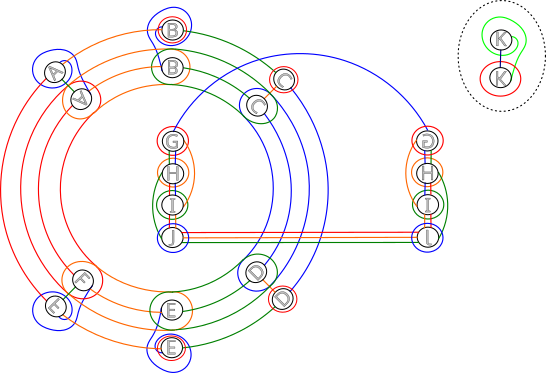}};
\end{tikzpicture}
\]
\caption{Quadrisection diagram for $\mathcal{S}_1(\mathbb{CP}^2)$, obtained from the trisection diagram for $\mathbb{CP}^2$ on the right \\ \footnotesize{Circles with matching labels are identified; the plane is compactified at infinity}}
\label{fig_spun_cp2}
\end{figure} 
\newpage
\printbibliography
\end{document}